\documentclass[11pt]{article}
\usepackage{amsmath,amsthm,amssymb,amsfonts,mathabx,MnSymbol, mathrsfs}
\usepackage[margin=1in]{geometry}
\usepackage{enumitem}
\usepackage{verbatim}
\usepackage{setspace}
\usepackage{color}
\usepackage{tikz}
\usepackage{natbib}
\usepackage{authblk}
\usepackage{physics}
\usepackage[title]{appendix}
\usepackage{subcaption}
\usepackage{tikz-3dplot}
\usepackage{pgfplotstable}
\usepackage{pgfplots}
\usepackage{float}
\usepackage[singlelinecheck=false]{caption}
\usepackage{authblk}

\usepackage{titlesec}
\titlespacing*{\section}{0pt}{0.15\baselineskip}{0.15\baselineskip}
\titlespacing*{\subsection}{0pt}{0.15\baselineskip}{0.15\baselineskip}

\newtheorem{corollary}{Corollary}
\newtheorem{proposition}{Proposition}

\newtheorem{definition}{Definition}

\DeclareMathOperator*{\argmin}{arg\,min}

\allowdisplaybreaks

\begin{document}

\title{\vspace{-2.5cm}Sequential Stochastic Optimization in Separable Learning Environments}
\author{\vspace{-0.4cm}
	R. Reid Bishop\thanks{\texttt{rebishop@coca-cola.com}, Data Science, The Coca-Cola Company} \quad
	Chelsea C. White III\thanks{H. Milton Stewart School of Industrial \& Systems Engineering, Georgia Institute of Technology}
}
\date{\vspace{-0.65cm} August 20, 2021}
\maketitle

\vspace{-0.9cm}
\begin{abstract}
	We consider a class of sequential decision-making problems under uncertainty that can encompass various types of supervised learning concepts.  These problems have a completely observed state process and a partially observed modulation process, where the state process is affected by the modulation process only through an observation process, the observation process only observes the modulation process, and the modulation process is exogenous to control.  We model this broad class of problems as a partially observed Markov decision process (POMDP).  The belief function for the modulation process is control invariant, thus separating the estimation of the modulation process from the control of the state process. We call this specially structured POMDP the \emph{separable POMDP}, or \emph{SEP-POMDP}, and show it (i) can serve as a model for a broad class of application areas, \emph{e.g.}, inventory control, finance, healthcare systems, (ii) inherits value function and optimal policy structure from a set of completely observed MDPs, (iii) can serve as a bridge between classical models of sequential decision making under uncertainty having fully specified model artifacts and such models that are not fully specified and require the use of predictive methods from statistics and machine learning, and (iv) allows for specialized approximate solution procedures. 
\end{abstract}

\vspace{-0.2cm}
\section{Introduction \& Literature Review}
\subsection{Introduction}
The complex stochastic, sequential decision-making environments that characterize reinforcement learning applications, in general, involve choosing between actions that greedily optimize over the immediate objective and actions that enable the decision-maker to learn about the environment in which they operate --- the well-known exploitation-exploration trade off. For the Markov decision process modeling (MDP) framework upon which these reinforcement learning applications are (typically) based, modelers often assume either (1) the uncertainty in the model is already captured by known and pre-specified transition probabilities (as in canonical operations research), or (2) the uncertainty is not modeled, but rather must be \emph{explored} by taking actions within the (real or simulated) environments.

For many applications in practice, however, there are different types of uncertainty --- endogenous uncertainty that the decision-maker can control and exogenous uncertainty that they cannot. For example, airlines must consider the weather when planning routes, investors must consider macroeconomic conditions when making investment decisions, and urgent, personalized therapeutics manufacturers must consider the patient's health when making production decisions. These types of decision-making environments, in which there is a \emph{separation} between types of uncertainty, are the focus of our investigation in this paper.

We introduce a sequential stochastic optimization model framework that is both an extension of the canonical MDP, and a special case of the generalized partially-observable MDP (POMDP), in which the uncertainty exhibits a \emph{separability} property --- some of the uncertainty in the system is affected by the actions of the decision-maker, and some of the uncertainty is not. Reminiscent of the Separation Principle in optimal stochastic control (\cite{Bismut78}, \cite{Tryphon13}), we call this class of models the \emph{separable POMDP}, or \emph{SEP-POMDP}. This modeling framework is widely applicable to many operations research problems and domains, for example: 
	\begin{itemize}
		\item  \vspace{-0.3cm} \textbf{Inventory management.} Constructing optimal inventory control policies under non-stationary demand (\cite{Treharne02}) and lost sales (\cite{Zipkin08}).
		\item \vspace{-0.3cm} \textbf{Finance.} Optimizing portfolio returns under stochastic volatility (\cite{Zhou09}) and mutual fund cash balancing (\cite{nascimento2010dynamic}).
		\item \vspace{-0.3cm} \textbf{Healthcare.} Constructing optimal policies for liver transplantation acceptance (\cite{Sandikci08}, \cite{Sandikci13}) and glycemic control for diabetes (\cite{JiangPowell15}).
	\end{itemize}
 \vspace{-0.3cm} 
\noindent
We summarize the main contributions of the paper, below:
\begin{enumerate}[label = (\arabic*)]
	\item \vspace{-0.3cm} We show that the SEP-POMDP inherits structural properties of the value function and optimal policy from analogous MDPs (\emph{e.g.} monotonicity, convexity, $L^\natural$-convexity, myopic optimal policies), under broad conditions.
	
	\item \vspace{-0.3cm} We show that the separability condition in the SEP-POMDP is flexible enough to incorporate many of the most popular statistics and machine learning models used in practice. These powerful supervised learning methods can be used to explain the exogenous uncertainty in the system. To our knowledge, this is a novel generalization that permits supervised learning models to be directly incorporated into the sequential stochastic optimization model. Since Markov decision processes form the foundation of much of reinforcement learning, this provides a bridge by which supervised learning and reinforcement learning might be connected in powerful ways. Moreover, the inherited structural properties in (1) are preserved when incorporating these supervised learning models in the SEP-POMDP.
	
	\item \vspace{-0.3cm} We discuss how structural properties of the value function and/or optimal policy that the SEP-POMDP inherits and separable supervised learning models might be used to construct specialized solution procedures that are tractable for large-scale applications. 
\end{enumerate}

\subsection{Literature Review}
The contributions, above, draw upon different fields of research. The research towards (1) is primarily inspired by \cite{Porteus75} and \cite{Smith02}. \cite{Porteus75} considered a notion of structure (which we adopt) as a restricted subspace of a function space in which every function in the subspace possesses some property of interest, and presented sufficient conditions by which a dynamic program has a value function and/or optimal policy function that are structured in this sense. We observe that structure has been useful for improved implementation and, as noted by \cite{Smith02}, in developing a qualitative understanding of the model and characterizing how the results will vary with changes in model parameters. For example, the optimality of a base-stock policy for a large class of inventory control models is easy to implement and has significant impact computationally. Further, \cite{Smith02} showed that for a MDP, if the reward function satisfies a property $\mathscr{P}$ and the transition probabilities satisfy a stochastic version of property $\mathscr{P}$, then the value function satisfies property $\mathscr{P}$, where structural properties that satisfy property $\mathscr{P}$ include monotonicity, convexity, supermodularity, combinations of these, and other properties of interest. We remark that, whereas \cite{Smith02} only considers value function structure, we consider optimal policy structure as well.

The most similar research to ours with respect to (2) is \cite{BertsimasMcCord19}, in which the authors consider multi-period stochastic optimization with ``side information". We show in Section 5 that this formulation is a special case of the SEP-POMDP, and the SEP-POMDP is flexible to incorporate many other supervised learning models in addition to that of \cite{BertsimasMcCord19}. Additionally, themes of incorporating Bayesian methods into reinforcement learning using POMDPs can be found in \cite{ross2011bayesian}, but whereas \cite{ross2011bayesian} considers primarily an approximate Bayesian reinforcement learning method for generalized partially observable decision-making environments, we consider separable learning environments in which supervised learning methods may be employed.

Finally, the research towards (3) is motivated by the well-known problem with POMDPs that the belief space is uncountably infinite, leading to computational complications. Various solution approaches from exact methods (\cite{Sondik73}, \cite{Sondik78}, \cite{Kaelbling98}), to fixed grid approximations (\cite{Lovejoy91}, \cite{Hauskrecht00}), to simulation-based approximations (\cite{Pineau03}, \cite{Spaan05}) have been proposed. We apply a solution procedure that utilizes base-stock optimal policy structure, support vector machines, and belief trajectory simulation to solve an inventory control problem under delayed procurement in Section 7. We also discuss other computational procedures that build upon the literature above, as well as information relaxation (\cite{Brown10}) and heuristics, in the appendix.

\subsection{Research Outline}
We now present an outline of the paper.  The formulation of the specially structured POMDP considered is presented in Section 2.  Section 3 presents preliminary results.  Key conditioning assumptions are given in Section 3.1, where the separability condition and SEP-POMDP are defined, and extensions of the Porteus results are given in Section 3.2.  The main structural results are presented in Section 4, where Sections 4.1 and 4.2 give value function and policy function structural results, respectively.  Thus far, the paper assumes that each of the model artifacts are fully specified.  In Section 5, we more realistically loosen this assumption, assuming some of these artifacts are better known than others.  We then show how the separability condition allows for the direct incorporation of many statistics and machine learning models into the SEP-POMDP formulation.  Discriminative learning blended with forecasting is the focus of Section 5.1 while Section 5.2 considers generative learning models.  Applications are presented in Section 6, indicating that the SEP-POMDP is a robust model that can describe many important real-world decision-making problems.  Computational solution approaches are discussed in Section 7 and an illustrative example is presented.  Conclusions are given in Section 8.

\section{Problem Formulation}
Consider a POMDP that has an infinite horizon and discrete decision epochs $t = 0, 1, \ldots$, and involves a completely observed state process $\{ s_t: t \geq 0 \}$ existing in a space $\mathcal{S} \subset \mathbb{R}^{d_\mathcal{S}}$, a partially observed modulation process $\{ \mu_t: t \geq 0 \}$ in a space $\mathcal{M} \subset \mathbb{R}^{d_\mathcal{M}}$, an observation process $\{ y_t: t \geq 1 \}$ in a space $\mathcal{Y} \subset \mathbb{R}^{d_\mathcal{Y}}$, and an action process $\{ a_t: t \geq 0 \}$ in a space $\mathcal{A} = \bigcup_{s \in \mathcal{S}} \mathcal{A}(s)$, where $a_t \in \mathcal{A}( s_t ), \forall t$. Assume that these processes are linked by the conditional probability $P[ y_{t+1}, s_{t+1}, \mu_{t+1} \vert s_t, \mu_t, a_t]$. It will be convenient for notational purposes to let $P[y_{t+1}, s_{t+1}, \mu_{t+1}| s_t, \mu_t, a_t] = P[y', s', \mu'| s, \mu, a]$.

We assume that $c: \mathcal{S} \times \mathcal{Y} \times \mathcal{A} \mapsto \mathbb{R}$ is the bounded single period cost function, where $c(s_t, y_{t+1}, a_t) = c(s, y', a)$ is the cost accrued during period $[t, t+1)$.  We further assume that the action at epoch $t$ can be selected on the basis of the information received up to $t$, $\mathscr{I}_t = \{s_t, s_{t-1}, \ldots,$ $s_0, y_t, y_{t-1}, \ldots ,$ $y_1, a_{t-1}, a_{t-2}, \ldots ,$ $a_0, \mathbf{b}_0 \}$, where $\mathbf{b}_0 = \{ \mathbf{b}_0(\mu): \mu \in \mathcal{M} \}$ is the prior distribution over $\mathcal{M}$.  A function mapping the set of all $\mathscr{I}_t$ into the set of all actions for all $t$ is a feasible policy.  The problem criterion is the expected total discounted cost over the infinite horizon, where we assume $\beta$, $0 \leq \beta < 1$ is the discount factor. The problem is to determine a feasible policy that minimizes the criterion with respect to all feasible policies. We note that though we present the results that follow for this infinite horizon formulation, the results can be suitably modified to the finite horizon case (where the horizon $T < \infty$), where the cost function is permitted to be dependent upon $t$, $c_t$, and we have a terminal cost function $c_T: \mathcal{S} \mapsto \mathbb{R}$.

\section{Preliminary Results}

Results in \cite{Sondik73} and \cite{Sondik78} imply that $\{( s_t, \mathbf{b}_t ), t \geq 0 \}$ is a sufficient statistic for this problem, where $\mathbf{b}_t = \{ \mathbf{b}_t (\mu): \mu \in M \}$ is the posterior belief distribution given the information up to time $t$, $\mathscr{I}_t$, namely that $\int_{\mathcal{M}} \mathbf{b}_t(\mu) \dd{\mu} = \mathbb{P}[\mu_t \in M \vert \mathscr{I}_t]$, $\forall M \subset \mathcal{M}$. We call $\mathbf{b}_t$ the Bayesian belief function at epoch $t$ and $\{ \mathbf{b}_t, t \geq 0 \}$ the belief function process. Let
\begin{align*}
\phi( y', s'| s, \mathbf{b}, a )		&= 		\int_{\mu'} \int_{\mu} \mathbf{b}(\mu) P[y', s', \mu' \vert s, \mu, a] \dd{\mu'} \dd{\mu}		\\
\lambda(\mu' \vert y', s', s, \mathbf{b}, a)			&= 	\frac{\int_{\mu} \mathbf{b}(\mu) P[y', s', \mu' \vert s, \mu, a] \dd{\mu}}{\phi(y', s' \vert s, \mathbf{b}, a)}, 
\quad \phi(y', s' \vert s, \mathbf{b}, a) \neq 0 		\\	
\lambda(y', s', s, \mathbf{b}, a)		&= 		\left\{ \lambda(\mu' \vert y', s', s, \mathbf{b}, a), \mu' \in M  \right\}.													
\end{align*}
We can think of $ \lambda(y', s', s, \mathbf{b}, a) $ as the posterior belief function $\mathbf{b}_{t+1}$, given $\mathbf{b}_t = \mathbf{b}, a_t = a, s_t = s, s_{t+1} = s'$, and $y_{t+1} = y'$. Similarly, $\phi( y', s'| s, \mathbf{b}, a )$ is the probability density of $y_{t+1}$ and $s_{t+1}$, given that $s_t = s$, $\mathbf{b}_t = x$, and $a_t = a$. Let $V$ be the Banach space of bounded value functions which map $\mathcal{S} \times \mathcal{B}$ into $\mathbb{R}$ endowed with the sup-norm, and let $H: V \mapsto V$ be defined as
	\begin{equation}
		Hv(s,\mathbf{b}) 	= 		\min_{a \in A(s)} \left\{ \mathbb{E} \left[c(s, y', a) \vert x \right] 	+ 	\beta \int_{y', s'} \phi(y', s' \vert s, x, a) 
			v \big( s', \lambda(y', s', s, \mathbf{b}, a) \big) \dd{y'} \dd{s'}  \right\},
		\label{eq: general optimality equation}
	\end{equation}
where $\mathbb{E} \left[c(s, y', a) \vert x \right] = \int_{y', s'} \phi(y', s' \vert s, \mathbf{b}, a) c(s, y', a) \dd{y'} \dd{s'}$. The optimality equation is $v = Hv$. Results from \cite{Puterman10} guarantee, by the contraction property of $H$, the existence of a unique value function, $v^*$, such that $v^* = Hv^*$, and that this fixed point is the expected total discounted cost accrued by an optimal policy. Further, we can restrict search for an optimal policy to $t$-invariant functions that select $a_t$ on the basis of $s_t$ and $\mathbf{b}_t$. Let $\Pi$ to be the space of such $t$-invariant functions from $\mathcal{S} \times \mathcal{B}$ to $\mathcal{A}$. The function, $\pi \in \Pi$ such that $\pi( s_t, \mathbf{b}_t ) = a_t$ causing the minimum in equation (\ref{eq: general optimality equation}) to be attained is an optimal policy. The expected total discounted cost accrued by this optimal policy can be attained by recursive application of $H$, so that $\lim_{n \to \infty} \norm{ v^* - v_n} = 0$, where $v_{n+1} = Hv_n$ for all $n$, given $v_0$ is any function in $V$, and $\norm{ \cdot }$ is the sup-norm.

\subsection{Key Conditioning Assumptions.} 
By the definition of conditional probability, \vspace{-0.3cm}
	\begin{equation*} 
		P[y', s', \mu'| s, \mu, a] = P[s'| y', \mu', s, \mu, a] P[y', \mu'| s, \mu, a].	 
	\end{equation*}
\vspace{-0.3cm} We assume that \vspace{-0.3cm}
	\begin{equation}
		\begin{split}
		P[s'| y', \mu', s, \mu, a] 	&= 		P[s'| y', s, a]	\\
		P[y', \mu'| s, \mu, a]			&= 		P[y', \mu'| \mu]. 
		\end{split} \vspace{-0.3cm} 
		\label{eq: SEP-POMDP conditioning assumption}
	\end{equation}
We call the POMDP presented in Section 2 with these key conditioning assumptions the \textit{separable POMDP}, or the SEP-POMDP.

We remark that the standard POMDP definition in the literature (\cite{Sondik73}, \cite{Sondik78}) assumes three processes, the partially observed state process, the observation process, and the action process, all of which are linked by the given probability $P [y', s' \vert s, a]$.  This standard definition assumes $P [y', s' \vert s, a] = P [y' \vert s', s, a] P [s' \vert s, a]$, where $P [y' \vert s', s, a]$ describes the relationship between the state, observation, and action processes and $P[s' \vert s, a]$ describes the controlled dynamics of the state process.  We note that the conditioning for the POMDP considered in this paper, $P [y', s', 
\mu' \vert s, \mu, a] = P [s' \vert y', s, a] P [y', \mu' \vert \mu]$, assumes that $s'$ is dependent on $y'$, rather than vice versa.  

Thus, for the SEP-POMDP we assume that the state process is affected by the modulation process only through the observation process, the observation process only observes the modulation process, and the modulation process is exogenous to control.  Under these assumptions, we can rewrite $\phi$,
	\begin{align*}
		\phi(y', s' \vert s, \mathbf{b}, a)  		&= 		\int_{\mu} \mathbf{b}(\mu) \int_{\mu'} p(s' \vert y', s, a) P[y', \mu' \vert \mu] \dd{\mu} \dd{\mu'}		\\
		&=		p(s' \vert y', s, a)  \int_{\mu, \mu'} \mathbf{b}(\mu) P[y', \mu' \vert \mu]	\dd{\mu} \dd{\mu'}		  \\
		&= 		p(s' \vert y', s, a) \sigma(y' \vert \mathbf{b}),
	\end{align*}
where we let $p(s' \vert y', s, a) = P[s' \vert y', s, a]$, and $\sigma(y' \vert x) = \int_{\mu, \mu'} \mathbf{b}(\mu) P[y', \mu' \vert \mu]$. We can then rewrite $\lambda$, by plugging in for $\phi$ and assuming $\phi(y', s' \vert s, \mathbf{b}, a) \neq 0$, as follows:
	\begin{align*}
		\lambda(\mu' \vert y', s', s, \mathbf{b}, a)			 &= 	\frac{\int_{\mu} \mathbf{b}(\mu) P[y', s', \mu' \vert s, \mu, a] \dd{\mu}}{\phi(y', s' \vert s, \mathbf{b}, a)}\\	
		&= 	 \frac{\int_{\mu} \mathbf{b}(\mu) P[y', s', \mu' \vert s, \mu, a] \dd{\mu}}{p(s' \vert s, y', a) \sigma(y' \vert \mathbf{b})} 		\\
		&= 	 \frac{ \int_{\mu} \mathbf{b}(\mu) P[y', \mu' \vert \mu] \dd{\mu} }{ \sigma(y' \vert \mathbf{b}) }	.		
	\end{align*}
Thus, $\lambda(\mu' \vert y', s', s, \mathbf{b}, a)$ is independent of $s', s, a$, and we denote $\lambda(\mu' \vert y', s', s, \mathbf{b}, a) = \lambda( \mu' \vert y', \mathbf{b})$ for all $\mu' \in \mathcal{M}$ and $\lambda(y',\mathbf{b}) = \{ \lambda(\mu' \vert y', \mathbf{b}), \mu' \in \mathcal{M} \}$. 

Note $\mathbb{E} [ c(s,y',a) \vert \mathbf{b}] = \int_{y', \mu'} \int_{\mu} P[y', \mu' \vert \mu] \mathbf{b}(\mu) c(s, y', a) \dd{y'} \dd{\mu'} \dd{\mu} = \int_{y'} \sigma(y' \vert \mathbf{b}) c(s, y', a) \dd{y'}$, and let
\[	h_{y'}(s, a, \bar{v})	=	c(s, y', a) + \beta \int_{s'}p(s' \vert y', s, a) \bar{v}(s') \dd{s'}.	\]

We then reformulate the operator $H$ as follows:
\[	Hv(s,\mathbf{b})	=	\min_{a \in \mathcal{A}(s)} \left\{	\int_{y'} \sigma(y' \vert \mathbf{b}) h_{y'} \big( s, a, v( \cdot , \lambda(y', \mathbf{b}) ) \big) \dd{y'}	\right\}. \]

We can now define the completely observed \emph{MDP analog} to the SEP-POMDP. Let $MDP_{y'}$ have single period cost function $c(s, y', a)$, transition structure $\{ p(s' \vert y', s, a) \}$, and operator 
	\begin{equation}	
		\bar{H}_{y'} \bar{v}(s) = \min_{a \in \mathcal{A}(s)} h_{y'}(s, a, \bar{v}).
		\label{eq: optimality equation of MDP analog}
	\end{equation}

We call the collection $\{ MDP_{y'}: y' \in \mathcal{Y} \}$ the completely observed MDP analog of the SEP-POMDP. We will seek to highlight the significance of this relationship to the MDP analogs later when we discuss conditions under which the SEP-POMDP inherits structural properties from the MDP analogs, but for now, we merely note that the observation realization in the MDP analog is a known quantity and functions as a parameter for the MDP. We might consider that $y'$ is a particular observable realization of the uncertainty in state dynamics for a traditional MDP. In the SEP-POMDP, this observation is permitted to be stochastic and, as we will see in Section 5, can be modeled using statistics and machine learning methods.

\subsection{The Porteus Results Extended} Let $V_\mathbf{b}$ denote the halfspace of $V$ induced by affixing $\mathbf{b} \in \mathcal{B}$ (\textit{i.e.} $V_\mathbf{b} = \{ f(\cdot, \mathbf{b}): f \in V \}$, $\forall \mathbf{b} \in \mathcal{B}$) and $\Pi_\mathbf{b}$ denote the halfspace of $\Pi$ induced by affixing $\mathbf{b} \in \mathcal{B}$. Suppose $\tilde{V}$ is a space of structured value functions $\mathcal{S} \mapsto \mathbb{R}$, and $\tilde{\Pi}$ is a space of structured Markovian deterministic policy functions $\mathcal{S} \mapsto \mathcal{A}$.

We now present the three structural conditions found in \cite{Porteus75} extended to the SEP-POMDP setting:
\begin{enumerate}[label = P(\alph*), leftmargin = 0.75in]
	\item Structured space of functions contains its limit points
	
	$\tilde{V}$ is a closed subset of $V_\mathbf{b}, \forall \mathbf{b} \in \mathcal{B}$.
	
	\item Structured Value Preservation
	
	$v(\cdot, \mathbf{b}) \in \tilde{V}, \forall \mathbf{b} \in \mathcal{B} 		\Rightarrow 		Hv(\cdot,\mathbf{b}) \in \tilde{V}, \forall \mathbf{b} \in \mathcal{B}$.
	
	\item Structured Policy Attainment
	\begin{align*}
	v(\cdot, \mathbf{b}) \in \tilde{V}, \forall \mathbf{b} \in \mathcal{B} \ 	\Rightarrow  \	&\exists \pi(\cdot, \mathbf{b}) \in \tilde{\Pi}, \forall \mathbf{b} \in \mathcal{B} \text{ s.t. } \\
	&Hv(\cdot, \mathbf{b}) = \int_{y'} \sigma(y' \vert \mathbf{b}) h_{y'} \big( \cdot, \pi(\cdot, \mathbf{b}), v(\cdot, \lambda(y',\mathbf{b}) ) \big) \dd{y'}, 	\forall \mathbf{b} \in \mathcal{B}. 
	\end{align*}
\end{enumerate}

We refer to P(a), P(b), and P(c) as the \textit{extended Porteus conditions}. Condition P(a) ensures that the limit point of a sequence of value functions obtained by the value iteration algorithm will be in the space of structured value functions, condition P(b) ensures that the structure of the value function is preserved when applying the dynamic programming operator $H$, and condition P(c) insures that for all structured value functions on $\mathcal{S}$, it suffices to search the space of structured policies (smaller than the space of all policies) for a $v$-improving policy.

We present a proposition in which we establish that P(a), P(b), and P(c) are sufficient conditions to guarantee that the value function and an optimal policy function are structured on $\mathcal{S}$. Subsequent results pertaining to structure on $\mathcal{S}$ demonstrate sufficient conditions for P(a), P(b), and P(c) to hold, by investigating the SEP-POMDP model primitives and the relationship to the MDP analog.
\begin{proposition}
	Assume the extended Porteus conditions hold. Then there exists a $\pi^*(\cdot,x) \in \tilde{\Pi}$ and a $v^*(\cdot, x) \in \tilde{V}$ for all $\mathbf{b} \in \mathcal{B}$ such that
	\[	v^*(s,\mathbf{b})	=	Hv^*(s,\mathbf{b})	=	\int_{y'} \sigma(y' \vert \mathbf{b}) h_{y'} \big( s, \pi^*(s,\mathbf{b}), v^*(\cdot, \lambda(y',\mathbf{b}) ) \big) \dd{y'}	\]
	for all $(s,\mathbf{b}) \in \mathcal{S} \times \mathcal{B}$.
	\label{Porteus}
\end{proposition}

Proof of the above result is a straightforward extension of Theorem 6.11.1 in \cite{Puterman10}. We remark that the structured optimal value function and the structured optimal policy are both modulated by the belief process $\{\mathbf{b}_t, t >0 \}$. The following corollary establishes that it is sufficient for only P(a) and P(b) to hold to establish structure of the value function on $\mathcal{S}$, absent structure in the policy.

\begin{corollary}
	If only P(a) and P(b) hold, then $v^*(\cdot, \mathbf{b}) \in \tilde{V}$ for all $\mathbf{b} \in \mathcal{B}$.
	\label{cor: Porteus}
\end{corollary}

\section{Main Structural Results}
%Propositions 2, 3, and 4 present sufficient conditions for guaranteeing value function structure on $S$ by guaranteeing that P(b) holds under operator $H$. 

We now present our primary structural results, which formalize the \textit{inheritance property} of SEP-POMDPs --- that value function \textit{and} optimal policy function structure of the MDP analog are inherited by the SEP-POMDP. Oftentimes in modeling efforts we make stylized and unrealistic simplifying assumptions for the sake of analytical tractability and gaining important qualitative intuition about a system (\textit{e.g.} demand is \textit{i.i.d.} across decision epochs, a firm operates independent of competitors). The thrust of the results in this section is that, for an important class of properties and models, we may analyze a simpler model and guarantee the structural properties hold for a more robust model. As we will see in later sections, this simpler model might assume \textit{e.g.} constant observations, and the structure of the optimal value function, or of an optimal policy, can still hold even under complex and sophisticated machine learning models for those observations. Thus, analytical tractability need not be traded for modeling realism.

\textbf{Preliminary definitions.} Before we state our inheritance proposition, we need to introduce two notions, as defined in \cite{Smith02}: \textit{C3 property} and its \textit{joint extension}.

\begin{definition}(C3 property)
	$\mathscr{P}$ is a closed convex cone property (C3) if and only if the set of all real-valued functions on $\mathcal{S}$ satisfying $\mathscr{P}$ forms a closed convex cone in the topology of pointwise convergence.
\end{definition}

Proposition 1 in \cite{Smith02} gives us an equivalent definition of C3 property in terms of an inequality ``test of satisfaction". A real-valued function $f$ on $\mathcal{S}$ satisfies a C3 property if and only if there exists a finite set of points $\{ s_j, j \in J_k \}$, $\{ s_i, i \in I_k \}$ and positive weights $\{ \gamma_j, j \in J_k\}$ and $\{ \gamma_i, i \in I_k\}$ such that
\[ \sum_{j \in J_k}	\gamma_j f(s_j)  	\leq 	\sum_{i \in I_k} \gamma_i f(s_i),    \quad 	\forall k \in K	\]
where $K$ is an index set.

Many structural properties, $\mathscr{P}$, of value functions with which we are interested in (\textit{e.g.} monotonicity, convexity) are \textit{C3} properties. The notion of the \textit{joint extension} of a \textit{C3} property allows us to extend the concept to real-valued functions on $\mathcal{S} \times \mathcal{A}$.
\begin{definition} (Joint Extension)
	Given a C3 property $\mathscr{P}$ on $\mathcal{S}$, a function $f: \mathcal{S} \times \mathcal{A} \mapsto \mathbb{R}$ satisfies a joint extension of $\mathscr{P}$ on $\mathcal{S} \times \mathcal{A}$, call it $\mathscr{P}^*$, if and only if for any $k \in K$, actions $\{ a_j, j \in J_k \}$, $\exists \{ a_i, i \in I_k \}$ such that 
	\[ \sum_{j \in J_k} \gamma_j f( s_j, a_j ) 	\leq 		\sum_{i \in I_k} \gamma_i f(s_i, a_i) \]
	where $\{ \gamma_j, j \in J_k \}$, $\{ \gamma_i, i \in I_k \}$ are finite sets of positive weights associated with the test of satisfaction for $\mathscr{P}$.
\end{definition}
The class of joint extensions of \textit{C3} properties includes subadditivity, L$^\natural$-convexity, joint submodularity, combinations of these, and others. It will be useful for us to note (especially in discussing separability, below) that all joint extensions of \textit{C3} properties are convex cones, in the sense that if $f$ and $g$ satisfy joint \textit{C3} property $\mathscr{P}^*$, then $\alpha f  + \beta g$ also has property $\mathscr{P}^*$, for $\alpha, \beta \in \mathbb{R}$.

\subsection{Structure on $\mathcal{S}$}
We begin by stating the Porteus conditions for MDPs, and recapitulating, for ease of reference, the structural implications for the MDP analog.
\begin{enumerate}[label = P$_{y'}$(\alph*), start = 2, leftmargin = 0.75in]
	\item Structured Value Preservation
	
	$\tilde{v} \in \tilde{V}		\Rightarrow 		\bar{H}_{y'}\tilde{v} \in \tilde{V}$.
	
	\item Structured Policy Attainment
	
	$\tilde{v} \in \tilde{V}	\Rightarrow  \exists \tilde{\pi} \in \tilde{\Pi} \text{ s.t. } \bar{H}_{y'}\tilde{v} = h_{y'} ( \cdot, \tilde{\pi}, \tilde{v} ). $
\end{enumerate}
The following proposition is due to \cite{Porteus75}; note Theorem 6.11.1 in \cite{Puterman10}.
\begin{proposition}
	Suppose P(a), P$_{y'}$(b), and P$_{y'}$(c) hold. Then there exists a $\pi^*_{y'} \in \tilde{\Pi}$ and a $v^*_{y'} \in \tilde{V}$ such that $v^*_{y'}(s) = \bar{H}_{y'}v^*_{y'}(s) = h_{y'}\big( s, \pi^*_{y'}(s), v^*_{y'} \big)$, for all $s \in \mathcal{S}$.
	\label{prop: structure MDP analog}
\end{proposition}

\begin{corollary}
	Suppose P(a) and P$_{y'}$(b) hold. Then $v^*_{y'} \in \tilde{V}$.
\end{corollary}

Suppose $\tilde{F}$ is a space of functions from $\mathcal{S} \times \mathcal{A}$ to $\mathbb{R}$ that is a convex cone. $\tilde{F}$ can be defined by a joint extension of a C3 property, $\mathscr{P}^*$, and thus encompasses the properties discussed in \cite{Smith02}. Further, let $\Delta$ be the space of feasible MDP analog policies from $\mathcal{S}$ to $\mathcal{A}$ (note that $\tilde{\Pi} \subseteq \Delta$). We present conditions by which the SEP-POMDP \textit{inherits} this MDP analog structure:
\begin{enumerate}[label = B(\alph*), start = 1, leftmargin = 0.75in]
	\item $\tilde{v} \in \tilde{V}		\Rightarrow 		h_{y'}(\cdot, \cdot, \tilde{v})\in \tilde{F}$
	
	\item $f \in \tilde{F}	\Rightarrow  \min\limits_{\delta \in \Delta} f^\delta \in \tilde{V}$
	
	\item $f \in \tilde{F} 	\Rightarrow 	\exists \tilde{\pi} \in \tilde{\Pi} $ s.t. $ \min\limits_{\delta \in \Delta}f^\delta = f^{\tilde{\pi}}$,
\end{enumerate}
where $f^\delta(s) = f(s, \delta(s))$ for all $s \in S$, and the minimum with respect to $\delta \in \Delta$ is taken pointwise, \textit{i.e.} $\left[ \min_{\delta \in \Delta} f^\delta \right](s) = \min_{a \in \mathcal{A}(s)} f(s, a)$ for all $s \in \mathcal{S}$.

Condition B(a) guarantees that, for the MDP analog, the function $h_{y'}$ is structured on $\mathcal{S} \times \mathcal{A}$. We recognize that this structure must be preserved under expectation in order for the fixed point of the optimality equation for the SEP-POMDP to inherit this structure, which is guaranteed in that $\tilde{F}$ is a space of functions that is a convex cone. Condition B(b) ensures that the minimization operation over feasible policies maps functions from $\tilde{F}$ into $\tilde{V}$. Finally, condition B(c) supposes we know, or can show, that minimizing functions of a certain structure on $\mathcal{S} \times \mathcal{A}$ yields a structured optimal policy. In fact, these conditions are quite mild, and hold for every one of the applications in Section 6. There are various results in the literature in this vein, \textit{e.g.} results pertaining to minimizing submodular functions on a lattice (\cite{Topkis78}) and minimizing $L^\natural$-convex functions (\cite{Zipkin08}). 

Note that B(a) and B(b) imply that P$_{y'}$(b) holds for all $y' \in \mathcal{Y}$, and B(a) and B(c) imply that P$_{y'}$(c) holds for all $y' \in \mathcal{Y}$. Thus, these are sufficient conditions for guaranteeing that the MDP analog is structured in its value function and an optimal policy by Proposition \ref{prop: structure MDP analog}. Our next proposition formalizes the \textit{inheritance property} of SEP-POMDPs by demonstrating that these sufficient conditions for guaranteeing structure for the MDP analog are, in fact, also sufficient for guaranteeing the SEP-POMDP is structured on $S$ in the same way. The proof follows by demonstrating that B(a), B(b), and B(c) are sufficient for guaranteeing that P(b) and P(c) hold, and then applying Proposition \ref{Porteus}.

\begin{proposition}
	Suppose P(a), B(a), B(b), and B(c) hold. Then there exists a $\pi^*(\cdot,\mathbf{b}) \in \tilde{\Pi}$ and a $v^*(\cdot, \mathbf{b}) \in \tilde{V}$ for all $\mathbf{b} \in \mathcal{B}$ such that
	\[	v^*(s,\mathbf{b})	=	Hv^*(s,\mathbf{b})	=	\int_{y'} \sigma(y' \vert \mathbf{b}) h_{y'} \big( s, \pi^*(s,\mathbf{b}), v^*(\cdot, \lambda(y',\mathbf{b}) ) \big) \dd{y'}	\]
	for all $(s,\mathbf{b}) \in \mathcal{S} \times \mathcal{B}$.
	\label{prop: SEP-POMDP inheritance}
\end{proposition}

The following is a straightforward corollary that shows that P(a), B(a), and B(b) are sufficient for guaranteeing value function structure, absent policy structure.
\begin{corollary}
	Suppose P(a), B(a), and B(b) hold. Then $v^*(\cdot, x) \in \tilde{V}$ for all $\mathbf{b} \in \mathcal{B}$.
	\label{cor: value function structure}
\end{corollary}

Of course, if the model primitives $p = \{p(s \vert y', s, a)\}$ and $c = \{c(s', y', a)\}$ are in spaces of structured transition probability functions, $\tilde{P}$, and cost functions, $\tilde{C}$, that guarantee that B(a) and B(b) hold, then the SEP-POMDP is structured in its value function by the Corollary \ref{cor: value function structure}.
\begin{corollary}
	Suppose P(a) holds, and that $p \in \tilde{P}$ for all $y' \in \mathcal{Y}$ and $c \in \tilde{C}$ for all $y' \in \mathcal{Y}$ imply that B(a) and B(b) hold. Then $v^*(\cdot, \mathbf{b}) \in \tilde{V}$ for all $\mathbf{b} \in \mathcal{B}$.
	\label{cor: value function structure with model primitives}
\end{corollary}

\subsection{Structure on $\mathcal{B}$}
%\textcolor{red}{I'm removing some of these propositions in favor of describing the results as extensions of POMDP structural results and referencing my dissertation. I want to leave the corollary and figure about the polyhedral partitioning of the belief space because I think it provides some intuition behind computational approximation procedures. I'll try to make this intuitive connection clear here.}

In this subsection, we discuss some known structural properties related to POMDPs, as they pertain to the SEP-POMDP when the spaces $\mathcal{S}$, $\mathcal{Y}$, $\mathcal{M}$, and $\mathcal{A}$ are discrete. The following proposition is due to \cite{Sondik73} and \cite{Sondik78}, in which successive value approximations achieved by applying the Bellman operator, $H$, preserve piecewise linearity and concavity of $v$ with respect to $\mathbf{b}$. Concavity is preserved in the limit. The proof of Proposition \ref{prop: piecewise linear and convex in x} can be found in \cite{Bishop19}.

%%% Piecewise linear and concave structure on \mathcal{B}.

\begin{proposition}
	The value function $v^*(s, \cdot)$ is concave in $\mathbf{b}$ on $\mathcal{B}$, for all $s \in S$.
	\label{prop: piecewise linear and convex in x}
\end{proposition}

If $v^*$ can be shown to be piecewise linear in $\mathbf{b}$ on $\mathcal{B}$ as well (such as if the optimal policy is finitely transient, as in \cite{Sondik78}), then we have a corollary result. For the standard POMDP model, the belief space $\mathcal{B}$ partitions into a finite number of convex, polyhedral regions that specify an optimal \textit{control} or \textit{action} to take. We note that for the SEP-POMDP, the belief space partitions into a finite number of convex, polyhedral regions that specify an optimal control or action \textit{for each} $s \in \mathcal{S}$. Thus, these non-overlapping regions in $\mathcal{B}$ specify a \textit{partial policy}, \textit{i.e.} functions from the state space $\mathcal{S}$ into the action space $\mathcal{A}$. If Proposition \ref{Porteus} holds, then these regions specify \textit{structured} partial policies.

\begin{corollary}
	Suppose $v^*$ is piecewise linear in $\mathbf{b}$ on $\mathcal{B}$. Then, there exists a partition of $\mathcal{B}$ into a finite number of convex, polyhedral regions $\{\mathcal{B}_j, j = 1, \ldots, n\}$ such that there exists a set of functions from $\mathcal{S}$ into $\mathcal{A}$, $\{\delta^*_j, j = 1, \ldots, n\}$, such that $\pi^*(\cdot, \mathbf{b}) = \delta^*_j$ for all $\mathbf{b} \in \mathcal{B}_j$, $j = 1, \ldots, n$.
	\label{cor: belief space partition}
\end{corollary}

These results can be utilized to motivate computational solution procedures. In the appendix, we discuss one way in which the belief space partition into a finite number of polyhedral regions specifying a \textit{structured} MDP analog policy (when the inheritance property of Propostion \ref{prop: SEP-POMDP inheritance} holds), $\delta^*$, can lead to computational efficiencies when utilizing the facet-generating algorithm in \cite{Sondik73}.

\textbf{Comment on additional structural properties of SEP-POMDPs.} For a more thorough compendium of structural properties of SEP-POMDPs --- including extensions of propositions in \cite{Smith02}, the value of information, sufficient conditions for monotone optimal policies with respect to the belief space, and inheritance under a functional description of dynamics --- we refer the reader to \cite{Bishop19}.

\section{Relationship to Supervised Learning}
%In our discussion thus far, we have assumed a proper instantiation of the SEP-POMDP model class, and shown attendant structural properties of the model class in this case. 

Thus far, we have assumed that each of the model artifacts --- the cost structure, definitions of the relevant processes, and transition probabilities of (\ref{eq: SEP-POMDP conditioning assumption}) --- are fully specified. In reality, some of these might be more confidently known than others. For example, suppose we are making inventory replenishment decisions for a single product, where $s_t$ is the inventory level, $y_{t+1}$ is the demand that arrives between $t$ and $t+1$, and $a_t$ is the replenishment amount. Suppose that replenishment is immediate and backlogging is permitted. In this system, we may be confident that $s_{t+1} = s_t + a_t - y_{t+1}$ accurately describes the dynamics of the inventory level, \emph{i.e.} that we can specify $P[s_{t+1} \vert s_t, y_{t+1}, a_t]$ from Equation (\ref{eq: SEP-POMDP conditioning assumption}), and that the relevant costs (\emph{e.g.} procurement, holding) are known. We may know that demand for our product is impacted in some way by the state of the market, $\mu_t$, but less certain how to specify the conditional demand and market distribution, $P[y_{t+1}, \mu_{t+1} \vert \mu_t]$. This is the situation that we consider in this section, in which the decision-maker seeks to model demand using the predictive methods from statistics and machine learning, sometimes in combination with ``domain expert" forecasts, based on historical observations of data pertaining to demand ($y_t$) and the state of the market ($\mu_t$). We show in this section how the formulation of the SEP-POMDP can encompass various types of learning models.

%In practice, however, it is often the case that the components of the system the modeler understands best, and can specify directly, pertain to the cost structure and the dynamics of the completely observed state process, $\{s_t: t \geq 0\}$. The uncertainties that arise due to the observation process, $\{y_t: t \geq 1\}$, and affect costs and the dynamics of $\{s_t: t \geq 0\}$, however, often must be predicted using methods from statistics and machine learning, trained on historical observational data. It is to this situation that we turn in this section, in which we show how the formulation of the SEP-POMDP can encompass various types of supervised learning models.

%Now that we have presented a general formulation of the SEP-POMDP, and attendant structural properties of the problem class, we will now show how the formulation can encompass various types of supervised learning models in practice. 

In each of these cases, assume that we have historical observations comprising a training dataset, $\mathscr{D} = \{(y^i, x^i ): i = 1, \ldots, N \}$. Here, $y^i$ indicates the $i$-th ``label", or realization of a target random variable (with support $\mathcal{Y}$, which we assume to be in $\mathbb{R}$ without loss of generality), that our machine learning models are principally interested in predicting, based on the realization of some observed auxiliary data vector $x^i$ (with support $\mathcal{X} \subseteq \mathbb{R}^{d_\mathcal{X}}$, dimension $d_{\mathcal{X}}$), and $N$ is the number of data points in our training data. We can choose to build these machine learning models to make predictions, at each time $t$, of $y_t$, using a combination of the (observed) auxiliary data $x_t$, latent (partially observed) variables $u_t$ (with support $\mathcal{U} \subseteq \mathbb{R}^{d_{\mathcal{U}}}$, dimension $d_{\mathcal{U}}$) that can represent either introduced model artifacts useful for describing the data or real characteristics of the data generating process for $(y_t, x_t)$, and additional modeling parameters $\theta_t$ (with support $\Theta \subseteq \mathbb{R}^{d_{\Theta}}$, dimension $d_{\Theta}$) that can be latent or known, time-varying or fixed.

For the SEP-POMDP, and in order to more fully describe its versatility and to better relate it to results in the machine learning literature, we consider the modulation process to be specified by these three types of machine learning model variables or parameters, that is $\mu_t = (x_t, u_t, \theta_t)$. Note that the SEP-POMDP assumption that $\mu_t$ is partially observed is an encompassing generalization for $(x_t, u_t, \theta_t)$ since the associated belief distribution $\mathbf{b}_t$ can simply assign probability 1 to the realization of whichever components are observed by the decision-maker (the completely observable case is a special case of the partially observable case). In the context of the above single product inventory replenishment problem, $x_t$ might represent related market data (\emph{e.g.} housing starts, consumer price index, Google searches for the product), $u_t$ might represent the ``underlying state of the market", and $\theta_t$ might represent model parameter values that are not completely known.  

%We identify the observation process of the SEP-POMDP with the labels of the supervised learning model. , \emph{i.e.} $y_t = y_t$ for all $t$.

%The vector $(y^i, x^i)$ represents a full realization of the observation process of the SEP-POMDP, $z^i$, and $N$ is number of data points in our training dataset. 

Recall, above, that a properly instantiated SEP-POMDP requires that we fully specify the following probability distribution:
	\begin{equation}
		P[y', \mu' \vert \mu] = P[y', x', u', \theta' \vert x, u, \theta].
		\label{eq: ML joint probability distribution}
	\end{equation}
In the context we consider here, this conditional joint distribution will be estimated using a statistical or machine learning model (or combination of models) in order to generate an approximate distribution that is ``close" to the true distribution, based on the training data $\mathscr{D}$, which we will denote $P_{\mathscr{D}}[y', x', u', \theta' \vert x, u, \theta]$. There are many different ways in which one might approach modeling this joint distribution, but in the context of the SEP-POMDP each of these fall under the two broad categories of \emph{generative} and \emph{discriminative} (plus, an associated Markov \emph{forecasting} model for auxiliary data) learning models, as in \cite{jebara2012machine}.

\textbf{\emph{Discriminative models, plus forecasting.}} In many cases, specifying a (generative, see below) model for the joint distribution over $(y_{t+1}, x_{t+1})$ can be difficult. In practice, many machine learning tasks are primarily concerned with making predictions about $y$, given some values of the auxiliary data $x$. These models are called \emph{discriminative} learning models. Since the SEP-POMDP is concerned with sequential decision-making environments, we require a full specification of the conditional joint distribution of $(y_{t+1}, x_{t+1})$. However, the modeler might choose to employ one of the many popular discriminative machine learning models in conjunction with a forecasting model for the auxiliary data process $\{x_t: t \geq 0\}$. There are many ways in which Equation (\ref{eq: ML joint probability distribution}) might decompose. One such decomposition is as follows, in which the forecasting model for the auxiliary data process is independent of observations of the $y$-process:
		\begin{equation*}
			P_{\mathscr{D}}[y', x', u', \theta' \vert x, u, \theta] = P_{\mathscr{D}}[y' \vert x', u', \theta', x, u, \theta] \cdot P_{\mathscr{D}}[x', u', \theta' \vert x, u, \theta].
		\end{equation*}
		
\textbf{\emph{Generative models.}} For generative models, the modeler specifies a model of the conditional joint distribution of $(y_{t+1}, x_{t+1})$, given values of the (possibly) latent $(u,\theta)$-process. In this setting, the joint distribution in Equation (\ref{eq: ML joint probability distribution}) decomposes, as follows:
		\begin{equation*}
			P_{\mathscr{D}}[y', x', u', \theta' \vert x, u, \theta] = P_{\mathscr{D}}[y', x' \vert u', \theta', x, u, \theta] \cdot P_{\mathscr{D}}[u', \theta' \vert x, u, \theta].
		\end{equation*}

In the subsequent subsections, we will discuss various learning models in the machine learning and optimization literature that fit within the SEP-POMDP framework. Though not a comprehensive list, the purpose of the discussion is to demonstrate substantial flexibility in incorporating learning models within the SEP-POMDP optimization models.

\subsection{Discriminative Learning, Plus Forecasting}
Many of the most popular supervised learning models, in practice, are aimed at some approximation of the conditional expectation, $\mathbb{E}[y_t \vert x_t = x]$, based on the historical training data, $\mathscr{D}$. At time $t$, predictions for future realizations of the target variable, $\{y_{t'}: t' > t\}$, will thus depend on forecasting future values of the auxiliary data, $\{x_{t'}: t' > t\}$. This is the setting we consider in this subsection, as we discuss how machine learning and forecasting models might be adapted and combined within the SEP-POMDP framework. Unless otherwise specified, in this subsection we will be principally concerned with learning models for specifying the following SEP-POMDP conditional probabilities:
	\begin{equation}
		P_{\mathscr{D}}[y', x', u', \theta' \vert x, u, \theta] = \underbrace{P_{\mathscr{D}}[y' \vert x', u', \theta', x, u, \theta]}_{\text{discriminative learning model}} \cdot \underbrace{P_{\mathscr{D}}[x', u', \theta' \vert x, u, \theta]}_{\text{forecasting}}.
		\label{eq: discriminative learning, plus forecasting conditional probabilities}
	\end{equation}
	
\textbf{\emph{Non-parametric machine learning, plus Markov forecasts.}} We begin by discussing the related work of \cite{BertsimasKallus20} and \cite{BertsimasMcCord19}, who consider the case of stochastic optimization ``with side information" (\cite{BertsimasKallus20} consider the single period case; \cite{BertsimasMcCord19}, the multi-period case). That is, they consider optimization problems in which decisions are made, given (possibly large-scale) auxiliary data that is useful for making predictions about uncertainties in the optimization problem. \cite{BertsimasKallus20} show how to use local, non-parametric learning methods, such as $k$-nearest neighbor ($k$-NN) regression, kernel regression, locally-estimated scatterplot smoothing (LOESS), classification and regression trees (CART), and random forests, trained on $\mathscr{D}$, to generate weight functions $\{w_{N,i}(x): i = 1, \ldots, N \}$ that approximate the following conditional probability for a fixed realization of the auxiliary data, $x^0$: %, evaluated over the finite training dataset $\mathscr{D} \subset \mathcal{Y} \times \mathcal{X}$:
	\begin{equation*}
		P[y = y^i \vert x', u', \theta', x = x^0, u, \theta] = P[y' = y^i \vert x = x^0] \approx w_{N,i}(x^0),
	\end{equation*}
where $w_{N,i}(x) \in [0,1]$ and $\sum_{i=1}^N{w_{N,i}(x)} = 1$, for all $x \in \mathcal{X}$. In \cite{BertsimasMcCord19}, they assume a Markov process for the auxiliary data (``side information"), and thus Equation \ref{eq: discriminative learning, plus forecasting conditional probabilities} is approximated as follows:
	\begin{equation}
		P_{\mathscr{D}}[y'=y^i, x'=x^i, u', \theta' \vert x, u, \theta] \approx w_{N,i}(x) \cdot P_{\mathscr{D}}[x' \vert x].
	\end{equation}
Since the model is non-parametric, and the discriminative learning models considered do not contain latent variables, when applying to the SEP-POMDP, the belief function $\mathbf{b}$ (defined as a probability distribution over $\mathcal{X}$) assigns probability 1 to the realization of $x_t$ at each decision epoch, $t$. The optimality equation in this case becomes:
	\begin{equation*}
		Hv(s, x) = \min_{a \in \mathcal{A}(s)} \left\{ \sum_{i=1}^{N} w_{N,i}(x) \left[c(s,y^i,a) + \beta \int_{x'} \int_{s'} P_{\mathscr{d}}[x' \vert x] p(s' \vert y^i, s, a) v(s', x') \dd{s'} \dd{x'}  \right]  \right\}.
	\end{equation*}
Note that since this is a special case of the SEP-POMDP --- that is, it satisfies the SEP-POMDP conditioning assumptions (\ref{eq: SEP-POMDP conditioning assumption}) --- the structural properties of the SEP-POMDP hold, including the inheritance property. Further, the discriminative learning models considered in \cite{BertsimasKallus20} and \cite{BertsimasMcCord19} (k-NN regression, kernel regression, LOESS, CART, random forests) are applicable to the SEP-POMDP, so long as they are accompanied by Markov forecasting model(s) for the auxiliary data process.

\textbf{\emph{Markov forecasting models.}} Given the importance of distributional forecasting for specifying the conditional probability (\ref{eq: discriminative learning, plus forecasting conditional probabilities}), it is worth considering the flexibility of the Markovian modeling assumption on the auxiliary data forecasting model $P_{\mathscr{D}}[x', u', \theta' \vert x, u, \theta]$. Notably, two important and broad classes of models that are popular in practice satisfy the Markovian assumption: Brownian motion-related stochastic processes (standard/geometric Brownian motion, Brownian motion with drift, Ornstein-Uhlenbeck processes, L\'evy processes, and multivariate extensions of these) and autoregressive time series models (auto-regression moving average, vector auto-regression). We include details pertaining to these in Appendix B.

Of course, other more direct Markov forecasting models for the auxiliary data also satisfy the forecasting conditioning assumption of (\ref{eq: discriminative learning, plus forecasting conditional probabilities}) --- for example, discrete-time Markov chains (DTMCs) as a model for $\{x_t: t \geq 0\}$ and, as is popular in practice, deterministic expert forecasts, such as forecasts for macroeconomic data published regularly by macroeconomists.

Finally, we note that, in the absence of auxiliary data, $\{x_t\}$, each of these forecasting models is directly applicable to the observation process, as well, by assuming another constructed ``auxiliary data" process, $\{\tilde{x}_t: t \geq 0\}$, such that $\tilde{x}_t$ is a finite history of the observation process, \emph{i.e.} $\exists \tau$ such that $\tilde{x}_t = [y_{t}, \ldots, y_{t-\tau}]$, with respect to which $(y_{t+1}, u_{t+1}, \theta_{t+1})$ satisfies the Markov property. Under this assumption:
	\begin{equation*}
		\begin{split}
			P_{\mathscr{D}}[y_{t+1}, \tilde{x}_{t+1}, u_{t+1}, \theta_{t+1} \vert \tilde{x}_t, u_t, \theta_t] &= P_{\mathscr{D}}[y_{t+1}, y_t, \ldots, y_{t-\tau+1}, u_{t+1}, \theta_{t+1} \vert y_{t}, \ldots, y_{t-\tau}, u_t, \theta_t] \\
			&= P_{\mathscr{D}}[y_{t+1}, u_{t+1}, \theta_{t+1} \vert y_{t}, y_{t-1}, \ldots, y_{t-\tau}, u_t, \theta_t]\\
			 &= P_{\mathscr{D}}[\tilde{x}_{t+1}, u_{t+1}, \theta_{t+1} \vert \tilde{x}_t, u_t, \theta_t].
		\end{split}
	\end{equation*}

\textbf{\emph{Other discriminative learning models.}} In addition to the discriminative learning models, above, other statistical learning methods fit within our framework. We will present the switching regression model of \cite{christiansen2020switching}, as an encompassing generalization of the Bayesian linear regression. For this time-dependent switching regression, there are assumed to be various ``regimes" (which we model with latent variable, $u_t$) under which the relationship of the observation $y_t$ to the auxiliary data $x_t$ is assumed to be captured by a different linear regression under each ``regime". These linear regressions are defined by the fixed parameters $\theta = \{ (\beta^0_u, \beta_u, \sigma_u): \forall u \in \mathcal{U}\}$, where $\beta^0_u$ is the scalar intercept, $\beta_u$ is the $d_{\mathcal{X}}$-dimensional vector of regression coefficients, and $\sigma_u$ specifies the standard deviation of the i.i.d. normally-distributed errors $\{ \varepsilon_{ut} \}$, that are assumed to be independent of the auxiliary data process. The time-dependence is assumed to captured by the latent variables, $\{u_t: t \geq 0\}$, which are assumed to followed a DTMC with transition probability distributions $\{ P_{\mathcal{U}}[\cdot \vert u]: u \in \mathcal{U} \}$. The auxiliary data are assumed to arise from i.i.d. draws from an unspecified probability distribution over $\mathcal{X}$, $P_{\mathcal{X}}$, and is independent of the $y$- and $u$-processes. Fully specified, the switching regression model is as follows:
	\begin{equation*}
		\begin{split}
			y_t = \sum_{u \in \mathcal{U}}(\beta^0_{u} + \beta_u \cdot x_{t} + \varepsilon_{ut}) \cdot \mathbf{1}\{u_t = u\} \\ 
			\varepsilon_{ut} \stackrel{\text{i.i.d.}}{\sim} \mathcal{N}(0, \sigma_u^2), \quad  u_t \sim P_{\mathcal{U}}[\cdot \vert u_{t-1}], \quad x_t \stackrel{\text{i.i.d.}}{\sim} P_{\mathcal{X}}. 
		\end{split}
	\end{equation*}
We can relate this to the discriminative learning model conditional probability condition of (\ref{eq: discriminative learning, plus forecasting conditional probabilities}):
	\begin{equation*}
	\begin{split}
			\underbrace{P_{\mathscr{D}}[y' \vert x', u', \theta', x, u, \theta]}_{\text{discriminative learning model}} &= P_{\mathscr{D}}[y' \vert x', u', \theta] \\ 
			&= \frac{1}{\sigma_u \sqrt{2 \pi}} \exp \left\{ -\frac{1}{2} \left( \frac{y' - \beta^0_{u'} - \beta_{u'} \cdot x'}{\sigma_{u'}} \right)^2 \right\} \\
			\underbrace{P_{\mathscr{D}}[x', u', \theta' \vert x, u, \theta]}_{\text{forecasting}} &= P_{\mathcal{U}}[u' \vert u] P_{\mathcal{X}}[x' \vert x]
	\end{split}
	\end{equation*}

Note that if there is assumed to be only one, static latent state, then the above model reduces to the standard Bayesian linear regression (when suitably equipped with prior distributions on the model parameters), the theory and analysis of which is well-documented in the literature (\cite{gelman2004bayesian}, \cite{west2006bayesian}).

\textbf{\emph{Discriminative learning models for specifying parametric distributions.}} We will now turn to another approach involving discriminative learning (for  the $y$-process), plus forecasting (for the $x$-process), with different underlying conditioning assumptions to Equation (\ref{eq: discriminative learning, plus forecasting conditional probabilities}), but nevertheless satisfying the SEP-POMDP conditioning assumptions in Equation (\ref{eq: SEP-POMDP conditioning assumption}). In this approach, the conditional probability for the $y$-process is specified by applying discriminative learning methods to estimating the parameters of a parametric probability distribution. Specifically, let us consider the probabilistic forecasting model and context of \cite{salinas2020deepar}, called ``DeepAR".

Suppose $\{y_t: t \geq 0\}$ is a vector-valued stochastic process of (possibly) high dimension, with components $y_{i,t}$. The auxiliary data process is ``assumed to be known for all time periods" (\cite{salinas2020deepar}) --- that is, at time $t$, we are assumed to have access to deterministic forecasts $\{\tilde{x}_{t'}: t' > t\}$ that can be global or associated with components of the $y$-process. For each $t$ and component $i$, $y_{i,t}$ is assumed to be drawn from a parametric distribution with likelihood function $P_y [\cdot \vert \tilde{\theta}_{i,t}]$, where the parameters specifying this distribution, $\tilde{\theta}_{i,t}$ (component $i$ of parameter vector $\tilde{\theta}_t$), are assumed to be functions of the outputs of a recurrent neural network (RNN) pertaining to component $i$ at time $t$. For instance, for real-valued $y_{i,t}$, this might be a Gaussian distribution, where the mean and standard deviation parameters are determined by the RNN output. We denote these RNN outputs, for each time $t$ and component $i$ as $u_{i,t}$, and the function specifying this relationship as $f_{\theta}$, which may include global parameters associated with the RNN, $\theta_h$ (and, thus, the overall parameter vector, $\theta_t = (\tilde{\theta}_t, \theta_h)$, consists of fixed and time-varying components): $\tilde{\theta}_{i,t} = f_{\theta}(u_{i,t}, \theta_h)$.
	%\begin{equation*}
	%	\tilde{\theta}_{i,t} = f_{\theta}(u_{i,t}, \theta_h).
	%\end{equation*}

The outputs, $u_{i,t}$, are modeled to be based on a RNN, a nonlinear function that we denote by $h$ and parametrized by $\theta_h$, taking as input the prior output, $u_{i,t-1}$, as well as the latest realization of the $i$-th component of the target variable, $y_{i,t}$, and the associated auxiliary data, $x_{i,t}$:
	\begin{equation*}
		u_{i,t} = h \left( u_{i,t-1}, y_{i,t}, x_{i,t}, \theta_h \right).
	\end{equation*}
In this model, the SEP-POMDP conditioning assumption in Equation (\ref{eq: SEP-POMDP conditioning assumption}) decomposes as follows:
	\begin{equation*}
		\begin{split}
			P_{\mathscr{D}}[&y_{t+1}, x_{t+1}, u_{t+1}, \theta_{t+1} \vert x_t, u_t, \theta_t] \\
			&=\prod_{i=1}^{d_{\mathcal{Y}}} P_y \big[ y_{i, t+1} \vert \tilde{\theta}_{i,t} \big] \cdot \mathbf{1}\big\{ \tilde{\theta}_{i,t+1} = f_{\theta} (u_{i,t+1}, \theta_h) \big\} \cdot \mathbf{1}\big\{ u_{i,t+1} = h \left( u_{i,t}, y_{i,t+1}, x_{i,t}, \theta_h \right) \big\} \cdot \mathbf{1}\{ x_{t+1} = \tilde{x}_{t+1} \}.
		\end{split}
	\end{equation*}
Note that, since we are assuming that the auxiliary data and forecasts are known, the auxiliary data forecasts are a special case of the Markov forecasting assumption described, above.

\subsection{Generative Learning}
Recall that the SEP-POMDP requires a full specification of the joint conditional probability in (\ref{eq: SEP-POMDP conditioning assumption}). Rather than specifying this distribution by decomposing it into parts and building various discriminative and forecasting learning models for these parts (as in the prior subsection), we might instead choose to model the joint distribution directly. We discuss two broad classes of these \emph{generative} learning models in this subsection.

\textbf{\emph{Hidden Markov Models.}} Hidden Markov models (HMMs) are a widely used and flexible generative learning model that has found applications in domains ranging from computational biology (\cite{eddy2004hidden}) to speech pattern recognition (\cite{rabiner1989tutorial}) to demand modeling in inventory systems (\cite{malladi2020dynamic}).

In the simplest formulation, HMMs are characterized by two discrete conditional probability distributions --- the Markov transition probabilities of the ``hidden" (latent) state process $\{u_t: t \geq 0 \}$, $\{P_{\mathcal{U}}[u' \vert u]: u', u \in \mathcal{U} \}$, and the probability distribution for the emissions $(y_t, x_t)$, $\{ P_{(y,x)} [y, x \vert u]: y \in \mathcal{Y}, x \in \mathcal{X}, u \in \mathcal{U} \}$. Thus, the SEP-POMDP conditioning assumptions are straightforwardly satisfied:
	\begin{equation*}
		P_{\mathscr{D}}[y', x', u', \theta' \vert x, u, \theta] = P[y', x' \vert u'] P[u' \vert u].
	\end{equation*}
This HMM formulation is extensible, for example to permit multivariate Gaussian emission distributions with parameters, $\theta_G$, specifying the mean and covariance structure:
	\begin{equation*}
	P_{\mathscr{D}}[y', x', u', \theta' \vert x, u, \theta] = P[y', x' \vert u', \theta_G] P[u' \vert u].
	\end{equation*}

\textbf{\emph{Bayesian networks.}} Another popular generative learning model is the \emph{Bayesian network}, which is a representation of joint probability distributions (often high-dimensional) using directed acyclic graphs in which edges represent local conditional dependencies (\cite{bishop2006pattern}). This generality of Bayesian networks as models of joint probability distributions, when applied to the joint distribution of $(y_t, x_t)$ in the SEP-POMDP, make them an encompassing generalization of the various modeling combinations that we have discussed in this section, above.

\textbf{\emph{On training the machine learning models and Bayesian updating.}} It might be clarifying, at this point, to discuss options regarding implementation of these machine learning models within our SEP-POMDP optimization model. In all cases, before we seek to solve our optimization problem, we first train out machine learning model(s) on the training dataset, $\mathscr{D}$, which gives us our joint distribution, $P_{\mathscr{D}}$. Once we proceed to solving our optimization problem, we may choose a variety of implementation methods.
	\begin{enumerate}
		\item \emph{Scoring the machine learning model.} In this option, we train the model before optimizing, affix the model parameters, $\hat{\theta}_{\mathscr{D}}$, and ``score the model" (as data science practitioners would say) --- that is, we do not re-estimate model parameters based on new observations once we have begun optimizing. The Bayesian inference mechanism, $\lambda$, is applied only for inferring latent variables, $u_t$, and \emph{not} model parameters.
		
		\item \emph{Bayesian model updating.} For certain types of statistical learning models, we may permit model re-training based on observed realizations of $(y_t, x_t)$ by including $\theta_t$ as a latent variable in the model and allowing Bayesian updating of the parameter(s) via $\lambda$. For example, in the case of a ``discriminative, plus forecasting" mode with Bayesian linear regression, we might permit posterior updates of the regression coefficients via $\lambda$.
		
		\item \emph{Online model updating.} Some machine learning models are not naturally suited to Bayesian model updating. For these types of models re-training based on observed realizations of $(y_t, x_t)$ --- updating, at time $t$, $P_{\mathscr{D}}$ based on $\mathscr{D} \cup \{(y_\tau, x_\tau): \tau < t\}$ --- must occur in the form of an iterative process of training the machine learning model and solving the SEP-POMDP.
	\end{enumerate}

\section{Applications}
In this section, we give some real-world examples of decision-making problems that fit within our SEP-POMDP framework --- following the examples of \cite{Treharne02} for inventory control, \cite{Sandikci13} for liver transplantation decisions, and \cite{Zhou09} for financial portfolio optimization. Additionally, we will discuss \cite{JiangPowell15}, as an example of how the inheritance property might usefully facilitate extensions of computational solution procedures and applications for MDPs to SEP-POMDPs.

\textbf{\textit{Inventory.}} Consider the inventory management context of \cite{Treharne02}, in which the decision-maker is a plant manager in charge of making regular inventory procurement decisions, $a_t$, in the face of economic uncertainty. At each procurement epoch, $t$, we know that our current inventory level is $s_t$. Suppose that we model that there is a state of the economy, $\mu_t$, for which we receive signals at each epoch through demand, $y_{t+1}$, that evolves independently of our procurement decisions --- in other words, inventory dynamics can be described by the conditional probability $P[s_{t+1} \vert y_{t+1}, s_t, a_t]$, and the demand and economic dynamics can be described as $P[\mu_{t+1}, y_{t+1} \vert \mu_t]$. Under this scenario, and suitable cost structures (e.g. the standard Newsvendor costs), \cite{Treharne02} prove that a non-stationary base stock policy, for which the base stock level at each epoch depends on a belief distribution over possible economic states, is optimal --- an inheritance result we could expect from Proposition \ref{prop: SEP-POMDP inheritance}.

\textbf{\textit{Liver Transplants.}} Now, consider the context of \cite{Sandikci13}, in which the decision-maker is an end-stage liver disease patient trying to optimize his or her decision to accept or reject offered potential liver transplants. The quality of the liver depends on the patient's unobserved ranking, $\mu_t$, on the United Network for Organ Sharing (UNOS) liver transplant list. At each decision epoch, $t$, the patient makes their decision, $a_t$, to accept or reject the offered liver on the basis of their known current health status, $h_t$, and the history of observed liver qualities, $\{l_t\}$, and published transplant list ranges on the UNOS website, $\{\omega_t\}$. The completely observed state component in this problem is $s_t = (h_t, l_t)$, the known current health status and liver quality. Observations of the true ranking on the UNOS transplant list are through the offered liver quality and published transplant list ranges, and thus can be described by the conditional probability $P[y_{t+1}, \mu_{t+1} \vert \mu_t]$, where $y_{t} = \{l_t, \omega_t\}$. In \cite{Sandikci13}, structural properties of an optimal policy are proven, such as the optimality of a control limit policy, which we could expect from Proposition \ref{prop: SEP-POMDP inheritance}.

\textbf{\textit{Financial Portfolio Optimization.}} Now consider, as in \cite{Zhou09}, that the decision-maker is seeking to optimize the value of his or her investment portfolio over a finite time period $[0, T]$ and under stochastic volatility conditions. For simplicity, assume that the decision-maker is managing a portfolio containing a single riskless asset with rate of return, $r$, and buy/sell decisions, $\{a_t: t \geq 0 \}$, are made at regular ``clock time" intervals of length $\varepsilon$ (that is, the clock time between each decision epoch $t$ and $t+1$ is $\varepsilon$). The model in \cite{Zhou09} considers that the asset price, $y_t$, evolves in continuous time according to geometric Brownian motion, the dynamics of which are governed by the following stochastic difference equation:
	\begin{equation*}
		y_{t+1} = x_t \text{exp} \left\{ \left( r -  \frac{u_{t+1}^2}{2} \right) \varepsilon + u_t \sqrt{\varepsilon} W^y_t \right\},
	\end{equation*}
where $u_t$ is the latent volatility at time $t$, $\{W^y_t: t \geq 0\}$ are i.i.d. Gaussian random variables, and $x_t = y_t$. The latent volatility process is assumed to be a mean-reverting process (with mean version parameter $\theta_{\text{mean}}$, mean reversion value $\theta_0$ and noise parameter $\theta_{\text{noise}}$), the dynamics of which can be approximated by:
	\begin{equation*}
		u_{t+1} = u_t + \theta_{\text{mean}} (\theta_0 - u_t) \varepsilon +  \theta_{\text{noise}} \sqrt{\varepsilon} W^u_t,
	\end{equation*}
$\{W^u_t: t \geq 0\}$ are i.i.d. Gaussian random variables independent of $\{W^y_t: t \geq 0\}$. Finally, the state, $s_t$, of the SEP-POMDP is the value of the portfolio at time $t$:
	\begin{equation*}
		\tilde{s}_{t+1} = (\tilde{s}_t - a_t x_t) e^{r \varepsilon} + a_t (y_{t+1} - x_t),
	\end{equation*}
with the objective being to maximize the expected value of $\tilde{s}_T$. For our purposes here, we consider $x_t$ to be represented as a completely observed component of the modulation process (as in Section 5), and also as a component of the state space, $s_t = (\tilde{s}_t, x_t)$. Note that the dynamics of this model satisfy the SEP-POMDP conditioning assumption in Equation (\ref{eq: SEP-POMDP conditioning assumption}).

\begin{comment}
Explicitly, the SEP-POMDP conditioning assumption (\ref{eq: SEP-POMDP conditioning assumption}) is satisfied as follows:
	\begin{equation*}
		\begin{split}
			P[s_{t+1} \vert y_{t+1}, s_t, a_t] &= \mathbf{1} \left\{ s_{t+1} = (\tilde{s}_t - a_t x_t) e^{r \varepsilon} + a_t (y_{t+1} + y_t) \right\} \\
			P[y_{t+1}, x_{t+1}, u_{t+1}, \theta_{t+1} \vert x_t, u_t, \theta_t] &= 
		\end{split}
	\end{equation*}
\end{comment}

\textbf{\textit{Monotone Approximate Dynamic Programming.}} Finally, we consider the MDP setting of \cite{JiangPowell15}, in which the authors demonstrate convergence of an approximate dynamic programming algorithm for solving MDPs in which the value functions are provably monotone on the state space $\mathcal{S}$. Incorporating knowledge of the monotone value function structure is demonstrated to substantially improve the computational tractability of the MDP models of selected applications in regenerative optimal stopping, energy storage and allocation, and glycemic control for diabetes. Each of these applications are shown to have monotone optimal value functions under conditions presented in Proposition 1 of \cite{JiangPowell15}. In Appendix B, we show that these conditions are sufficient for the SEP-POMDP inheritance of this monotone value function structure under Corollary \ref{cor: value function structure}. 

What is the significance of this inheritance? Each of the applications considered in \cite{JiangPowell15} satisfy these conditions, and thus, there exist SEP-POMDP extensions of these models that preserve monotone optimal value functions. An important extension, in light of Section 5 and discussed in Appendix B, is in building statistical learning models for explaining the stochasticity in state dynamics present in each of these applications, based on auxiliary data. For example, in their energy storage and application example, the decision-maker is seeking to maximize revenues while producing and transferring energy across the energy storage network, as well as purchasing energy from the spot market. These decisions are inextricably linked to the uncertain energy demand on the system. A SEP-POMDP formulation of the problem might include a statistical learning model for \emph{predicting} demand based on seasonal patterns, weather data, Google search data, energy prices in the market, \emph{etc.} We are guaranteed by the inheritance property, that including such a predictive demand model would preserve monotonicity, and thus the methods of \cite{JiangPowell15}, and their attendant computational benefits, for determining an optimal policy are still applicable. More broadly, this is but one example of a set of conditions guaranteeing monotone optimal value functions for applications of MDPs. For other conditions, and resulting applications, a similar connection to the monotone approximate dynamic programming method of \cite{JiangPowell15} might possibly be established.

\section{Computational Example}
There are many different approaches we might take to solving the SEP-POMDP, including specialized approaches that utilize the structural properties we have discussed: notably inheritance and separable learning. We discuss one approach based on simulating belief trajectories, that we then combine with inheritance in solving an inventory problem with time-delayed replenishment. We discuss other computational methods, including exact methods in which we discuss the computational benefits that might be gained by exploiting the relative tractability of the MDP analogs compared to the generalized POMDP, approximate methods based on information relaxation, and heuristics in Appendix E.

%\subsection{Computational Example: Single Product Inventory Replenishment Under Procurement Delays}
We now give an example of how a modeler might combine various structural properties of the SEP-POMDP to generate ``good" policies. There are many ways (and it present an interesting direction for future research) in which specialized solution procedures for the SEP-POMDP could be developed, so this is example is but one of many and its inclusion is meant for illustrative purposes, as a concrete example of how inheritance and separability can be used in a computational solution procedure. This example pertains to inventory management, and it constructs ``good" policies in a solution procedure that: (1) utilizes a belief trajectory simulation method, as in Appendix E, (2) constructs partitions of the belief space, $\mathcal{B}$, using support vector machines, and (3) incorporates a generative learning model for demand, as in Section 5.2. The discussion in this section is based on \cite{Bishop19} chapter 3. We keep the discussion necessarily brief, and refer the reader there for a more detailed presentation, including additional results and a more extensive computational study.

\emph{\textbf{Formulation.}} Consider that the decision-maker is making inventory replenishment decisions for a single product over time, in which replenishment decisions made at decision epoch $t$ are realized at decision epoch $t+\tau$ (modeling, \emph{e.g.}, procurement procurement delays). We model this as a SEP-POMDP with the following constituent processes:

\begin{itemize}
	\item \vspace{-0.3cm} $\{s_t: t = 0, 1, \ldots\}$ is defined to be the \textit{inventory level process}, where $s_t$ is the inventory level at the decision epoch $t$ prior to satisfying demand and being replenished.
	
	\item \vspace{-0.3cm} $\{y_t: t = 1, 2, \ldots\}$ is defined to be the \textit{demand process}, where $d_t$ is the demand that becomes known just before decision epoch $t$. The support for the demand process is assumed to be finite, $|\mathcal{Y}| < \infty$.
	
	\item \vspace{-0.3cm} $\{a_t: t = 0, 1, \ldots\}$ is the \textit{replenishment process}, where $a_t$ is the replenishment decision made at decision epoch $t$.
	
	\item \vspace{-0.3cm} $\{x_t: t = 1, 2, \ldots\}$ is the \textit{additional observation data (AOD) process}, where $x_t$ represents data that becomes known just before epoch $t$ from sources in addition to demand that might be useful in more accurately forecasting demand. The set of all possible observations is $\mathcal{X}$ and is assumed to be finite. We assume that $\{x_t: t \geq 1\}$ is completely observed, as in Section 5.
\end{itemize}
\vspace{-0.3cm} In this SEP-POMDP, we will train a (generative) hidden Markov model for the joint demand and AOD processes, $\{(y_t, x_t): t \geq 0 \}$, with latent state process $\{u_t: t \geq 0 \}$, as a model for the following SEP-POMDP conditional probability:
	\begin{equation*}
		P[y_{t+1}, \mu_{t+1} \vert \mu_t] = P[y_{t+1}, x_{t+1} \vert u_{t+1}] P[u_{t+1} \vert u_t].
	\end{equation*}

The costs at time $t$ will be accrued upon realization of the inventory order, according to the familiar Newsvendor cost function: $c a_t + h(s_t + a_{t-\tau} - y_{t+1})^+ + p (y_{t+1} - s_t - a_{t-\tau})^+$, where $(b)^+ = \max(0, b)$. The per-unit holding cost is $h$, the per-unit purchase cost is $c$, and $p$ is the per-unit underage cost. Further, we assume that the inventory, demand, and replenishment processes are related through the stochastic difference equation $s_{t+1} = s_t + a_{t - \tau} - y_{t+1}$, which assumes backlogging is allowed, where $\tau$ is the replenishment delay. This equation can be described as a conditional probability $P[s_{t+1} \vert s_t, y_{t+1}, a_{t - \tau}]$. \begin{comment} This definition of state dynamics differs slightly from the usual formulation in the inventory literature since we allow the decision maker to know the demand at epoch $t$, rather than the typical assumption that $y_t$ represents a random variable demand realized between epochs $t$ and $t+1$. This permits the formulation here to capture both build-to-order  production environments ($\tau = 0$) and also more traditional production environments $\tau > 0$. \end{comment} This conditional probability is atypical since the inventory dynamics depends on lagged decisions, rather than the actions at the current decision epoch. In the formulation, below, we re-frame the optimality equation as a SEP-POMDP with \emph{inventory position} as state.

In this formulation, the decision-maker at epoch $t$ chooses the total amount of inventory possessed through the interval $[t, t+\tau]$, $\tilde{a}_t \triangleq s_t + \sum_{j=1}^{\tau}a_{t-j} + a_t$ (note that $s_{t+\tau} = \tilde{a}_t - a_t - \sum_{j=1}^{\tau}y_{t+j}$). If we let $\tilde{s}_t = \tilde{a}_t - a_t$ be the \textit{inventory position} through interval $[t, t + \tau]$ before ordering, then we have that $\tilde{s}_{t+1} = \tilde{s}_t + a_t - y_{t+1}$, which is familiar as the inventory difference equation under backlogging. Additionally, we can project out purchase costs in the resulting optimality equation is $v = \tilde{H}v$, where $\tilde{H}$ is defined to be:
	\begin{equation}
		\begin{split}
			\tilde{H} v(\tilde{s}, \mathbf{b}) &= \min_{\tilde{a} \geq \tilde{s}} \left\{\mathbb{E} \left[  \tilde{h} \left(\tilde{a} - \sum_{j=1}^{\tau} y_j \right)^+ + \tilde{p} \left(\sum_{j=1}^{\tau} y_j  - \tilde{a}\right)^+ \vert \mathbf{b} \right] + \beta  \sum_{y', x'} \sigma(y', x' \vert \mathbf{b}) v \big(\tilde{a} - y', \lambda(y', x', \mathbf{b}) \big) \right\}, 
		\end{split}
		\label{eq: opt d}
	\end{equation}
and where $\tilde{h} = \beta^\tau h + c$ and $\tilde{p} = \beta^\tau p - c$. With a little abuse of notation, we use $\sum_{j=1}^\tau y_j$ to denote the (random variable) sum over the next $\tau$ realizations of the demand process, \emph{i.e.} at decision epoch $t$, the sum over $y_{t+1}, y_{t+2}, \ldots, y_{t+\tau}$. The distributions $\sigma$, $\lambda$ are defined as in Section 3. For further details regarding this formulation, we refer the reader to \cite{Bishop19}.

For canonical single-product inventory problems modeled as MDPs, base stock policies are well-known to be optimal. Proposition \ref{prop: base stock optimality} uses the inheritance property of SEP-POMDPs to prove that a base stock policy is optimal for this problem setting under a HMM learning model for demand, with base stock levels, $\{a^*(\mathbf{b}): \mathbf{b} \in \mathcal{B}\}$, defined as the smallest (and hence unique) myopic minimizer such that:
	\begin{equation}
		a^*(\mathbf{b}) \in \argmin_{\tilde{a}} \left\{ \mathbb{E} \left[  \tilde{h} \left(\tilde{a} - \sum_{j=1}^{\tau} y_j \right)^+ + \tilde{p} \left(\sum_{j=1}^{\tau} y_j  - \tilde{a}\right)^+ \vert \mathbf{b} \right] \right\}. 
		\label{eq: base stock policy definition}
	\end{equation}

\begin{proposition}
	Suppose $a^*(\mathbf{b}) - y' \leq a^* \big( \lambda(y', x', \mathbf{b}) \big)$ for all $y', x', \mathbf{b}$. Then the $\tau$-lookahead policy, $\pi(\tilde{s}, \mathbf{b}) = \max\{ a^*(\mathbf{b}) - \tilde{s}, 0 \}$ for all $\tilde{s}, \mathbf{b}$ is optimal.
	\label{prop: base stock optimality}
\end{proposition}

The proof of Proposition \ref{prop: base stock optimality}, based on inheritance of myopic optimal policy structure from the MDPs of \cite{Sobel81}, is in Appendix D.

\emph{\textbf{Solution Procedure.}} Let $\Delta \triangleq \{\sum_{j=1}^\tau y_j: y_1, \ldots, y_\tau \in \mathcal{Y}\} = \{\delta_1, \ldots, \delta_{|\Delta|}\}$, the set of possible total demands over $\tau$ epochs, and suppose the $\delta_i$ are in ascending order $(\delta_1 < \delta_2 < \ldots < \delta_{|\Delta|})$. \cite{Bishop19} show that the optimal base stock levels induce a \emph{linear partition} of the belief space, $\mathcal{B}$ into sets $\{\mathcal{B}_\delta: \delta \in \Delta\}$ such that for all $\mathbf{b} \in \mathcal{B}_\delta$, $a^*(\mathbf{b}) = \delta$. These sets are defined by the Newsvendor critical fractile, $\frac{\tilde{p}}{\tilde{p} + \tilde{h}}$:
	\begin{equation}
		\mathcal{B}_{\delta_m} \triangleq  \left\{ \mathbf{b} \in \mathcal{B}: P \left[ \sum_{j=1}^\tau y_j  \leq \delta_{m-1} \vert \mathbf{b} \right]	< \frac{\tilde{p}}{\tilde{p} + \tilde{h}} 	\leq  P \left[ \sum_{j=1}^\tau y_j  \leq \delta_{m} \vert \mathbf{b} \right]	\right\}.
	\end{equation}
Rather than solving for these partitioning hyperplanes analytically, which can be difficult depending on the demand model, we construct them using Monte Carlo simulation and soft-margin support vector machines (SVM). The procedure is detailed in Figure \ref{fig: partition B}. In Step 1, we generate a finite grid of belief vectors through belief trajectory simulation. Then, in Step 2 and 3, we use Monte Carlo simulation of the demand process to calculate the estimated optimal base stock levels. These then serve as labels upon which we can train SVM classifiers in step 4. We note that the multi-class SVM of Step 4 can be solved by solving $|\Delta|$ one-versus-rest SVMs. Figure \ref{fig: all partition grids} illustrates this method for approximating the partition $\{\mathcal{B}_\delta: \delta \in \Delta\}$ for a small example. 

\emph{\textbf{Computational Experiments.}} Now we give an numerical example that is meant to be illustrative of the process a practitioner might go through to train a statistical learning model for demand, given historical observations of the demand and AOD processes, and then utilize this learning model to construct ``good" policies using the SVM-based method, above. For this example, we assume that the true demand and AOD processes are generated from a HMM. The dynamics of the latent states under the ``true" HMM ($\mathcal{H}^\text{true}$) are defined by the following transition matrix:
	\begin{equation*}
		U = 
		\begin{bmatrix}
			0.7 & 0.2 & 0.1 \\
			0.3 & 0.5 & 0.2 \\
			0.3 & 0.3 & 0.4 
		\end{bmatrix}, \quad P[u_{t+1} = j \vert u_t = i] = U(i, j).
	\end{equation*}
For each of these three latent states, the demand and AOD processes are drawn from discrete multi-variate Normal distributions, so that the conditional probabilities $P[y_t, x_t \vert u_t]$ are defined by the following mean ($\zeta_u$) vectors and covariance matrices ($\Sigma_u$):
	\begin{equation*}
		\zeta_1 = \begin{bmatrix}
		10  \\
		8
		\end{bmatrix},
		\Sigma_1 = \begin{bmatrix}
		5 & 1 \\
		1 & 5
		\end{bmatrix},
		\zeta_2 = \begin{bmatrix}
		20  \\
		10
		\end{bmatrix},
		\Sigma_2 = \begin{bmatrix}
		10 & 1 \\
		1 & 10
		\end{bmatrix},
		\zeta_3 = \begin{bmatrix}
		25  \\
		12
		\end{bmatrix},
		\Sigma_3 = \begin{bmatrix}
		15 & 1 \\
		1 & 15
		\end{bmatrix}.
	\end{equation*}
The other parameters specifying the SEP-POMDP inventory model are $\beta = 0.93$, $\tau = 2$, $\tilde{p} = 3$, $\tilde{h} = 1$. We simulate the policies across a horizon $T=65$. The numerical experiment proceeds as follows:
	\begin{enumerate}
		\item Initialize $\mathcal{H}^\text{true}$. Compute the SVM-generated base stock policy according to the procedure in Figure \ref{fig: partition B}, $\mathrm{SVM}^{\text{true}}$. Evaluate $\mathcal{H}^\text{true}$ according to the Monte Carlo policy evaluation procedure in Figure \ref{fig: SVM-Monte Carlo evaluation method} (with $\mathcal{H}^\text{eval} = \mathcal{H}^\text{true}$).
		
		\item \vspace{-0.3cm} Generate a synthetic training dataset, $\mathscr{D}$, by simulating multiple trajectories of length $T=65$.
		
		\item \vspace{-0.3cm} Train a HMM on $\mathscr{D}$ using the expectation maximization algorithm of \cite{baum1966statistical}, $\mathcal{H}^\text{train}$. Compute the SVM-generated base stock policy according to the procedure in Figure \ref{fig: partition B}, $\mathrm{SVM}^{\text{train}}$. Evaluate $\mathcal{H}^\text{train}$ according to the Monte Carlo policy evaluation procedure in Figure \ref{fig: SVM-Monte Carlo evaluation method} (with $\mathcal{H}^\text{eval} = \mathcal{H}^\text{train}$).
	\end{enumerate}
As in Section 5, for our example here we have (synthetically-generated) training $\mathscr{D}$ upon which we can train a learning model \emph{prior to} implementing (or ``scoring" the learning model) in the SEP-POMDP optimization problem. The solution procedure makes use of the policy structure (inheritance) and also separability (in belief simulation and HMM training) in order to construct good policy solutions.

For our computational experiment, policies are evaluated using $10,000$ Monte Carlo simulations. In Figure \ref{fig: computational example graph} we compare the evaluation of the base stock policy based on $\mathrm{SVM}^{\text{true}}$ to $\mathrm{SVM}^{\text{train}}$ for different sizes of the dataset $\mathscr{D}$. Since the expectation maximization algorithm used to train $\mathcal{H}^\text{train}$ does not have convergence guarantees, for each dataset size we give the HMM training 5 different random initializations and report both the policy evaluation under the best performing initialization and also the average across the initializations. Since we do not have convergence guarantees in training these HMMs, we see that the gap between the policy evaluations narrows as the training dataset size increases, but then plateaus.

\section{Conclusion}
We have introduced a specially structured POMDP, the SEP-POMDP, for modeling sequential decision-making environments in the presence of exogenous observations that affect the dynamics and objective of the system. We showed that this class of models inherits optimal value and policy function structural properties from related MDPs, thus extending the deep operations research literature proving such structures for the general MDP and also myriad real-world applications. In a particularly important discussion, we then showed that our formulation encompasses a wide array of supervised learning models for modeling the exogenous uncertainty introduced to the system through the observation process. The range of supervised learning methods is vast and includes: discriminative learning models such as random forests, LOESS, kernel regression, switching regressions, and autoregressive recurrent neural networks; Markovian forecasting models such as Brownian motion, Ornstein-Uhlenbeck processes, and ARMA processes; as well as generative models such as HMMs and Bayesian networks. We gave a sense for the range of applications for which the SEP-POMDP framework can include by discussing its relationship to models from various fields. Finally, we discussed a particular inventory problem under procurement delays, as an illustrative example as to how one might integrate various properties of the SEP-POMDP in a solution procedure. We give additional attention to computational considerations in the appendix.

Much of the reinforcement learning literature is concerned with learning (near) optimal policies through repeated interaction with the decision-making environment, and in many applications in a model-free environment. Developing these methods for learning in the midst of uncertainty is a natural evolution from the foundational MDP that arose out of the operations research community, in which assumptions that the transition probabilities in the system are well-specified are common. What happens, however, when interactions in the environment are expensive, or reinforcement learning requires a number of interactions that pushes the limits of our computing capabilities, as we seek to apply these methods to more and more complex real-world systems? Our reinforcement learning models could benefit substantially by leveraging supervised learning methods for modeling exogenous uncertainty in the system. We see the SEP-POMDP as a potentially foundational modeling framework for building next generation reinforcement learning methods and applications that leverage supervised learning for explaining the uncertainty in the system based on (possibly very large) data.

%There are many exciting areas for subsequent research, the most interesting of which to us is in exploring the connection between supervised learning methods for explaining exogenous uncertainty, based on (possibly big) data, and reinforcement learning, with consideration to the structural results that characterize the canonical operations research literature on MDPs. We hope that the SEP-POMDP framework we study in this paper, can serve as a bridge for facilitating new methods and applications that draw from the operations research, reinforcement learning, and supervised learning communities.

\begin{appendices}
	% NEW APPENDIX SECTION
	\section{Proof of Inheritance Property}
	
	\begin{proof}[Proof of Proposition \ref{prop: SEP-POMDP inheritance}.]
		We proceed by demonstrating that P(b) and P(c) hold and then applying Proposition \ref{Porteus}. Suppose $v(\cdot, x) \in \tilde{V}$ for all $\mathbf{b} \in \mathcal{B}$. Recall, we have
		\[ Hv(s,\mathbf{b}) 		= 		\min_{a \in \mathcal{A}(s)} \int_{y'} \sigma(y' \vert \mathbf{b}) h_{y'} \big( s, a, v(\cdot, \lambda(y',\mathbf{b})) \big) \dd{y'}.	\]
		By B(a), we have that $h_{y'} \big( \cdot, \cdot, v(\cdot, \lambda(y',\mathbf{b})) \big) \in \tilde{F}$ for all $(y', \mathbf{b}) \in \mathcal{Y} \times \mathcal{B}$. Further, 
		\[\int_{y'} \sigma(y' \vert \mathbf{b}) h_{y'} \big( \cdot, \cdot, v(\cdot, \lambda(y',\mathbf{b})) \big) \dd{y'}  \in \tilde{F}\] 
		as well, since $\tilde{F}$ is a space of functions that is a convex cone. \begin{comment} possessing a joint C3 property, $\mathscr{P}^*$.\end{comment} By B(b), minimizing over feasible policies from $S$ to $A$ maps functions in $\tilde{F}$ into $\tilde{V}$. We conclude that $Hv(\cdot, \mathbf{b}) \in \tilde{V}$ for all $\mathbf{b} \in \mathcal{B}$ and P(b) holds. 
		
		By the same logic, since $\int_{y'} \sigma(y' \vert \mathbf{b}) h_{y'} \big( \cdot, \cdot, v(\cdot, \lambda(y',\mathbf{b})) \big) \dd{y'}  \in \tilde{F}$, B(c) guarantees that P(c) holds as well. The conclusion follows by Proposition \ref{Porteus}.
	\end{proof}
	
\end{appendices}

% REFERENCES
\bibliography{refs_POMDPs}

\section*{Figures}
\begin{figure}[H]
	\centering
	\resizebox{0.6\textwidth}{2.5in}{
		\begin{tikzpicture}
		\draw (-4,0) node[anchor=north]{$(1, 0, 0)$}
		-- (4,0) node[anchor=north]{$(0, 1, 0)$}
		-- (0,8) node[anchor=south]{$(0, 0, 1)$}
		-- cycle;
		
		\draw (1,0) -- (3, 2);
		\draw (0,0) -- (0, 3);
		\draw (0, 3) -- (2.23, 3.5);
		\draw (-3, 2) -- (0, 3);
		\draw (-1.5, 5) -- (0,3);
		
		\node [align=flush center,text width=8cm] at (-1.7, 1.1)
		{$\delta^*_1$};
		\node [align=flush center,text width=8cm] at (-1.5, 3.3)
		{$\delta^*_2$};
		\node [align=flush center,text width=8cm] at (0.2, 4.9)
		{$\delta^*_3$};
		\node [align=flush center,text width=8cm] at (1.1, 1.8)
		{$\delta^*_4$};
		\node [align=flush center,text width=8cm] at (2.75, 0.7)
		{$\delta^*_5$};
		\end{tikzpicture}
	}
	
	\caption{A graphical depiction of Corollary \ref{cor: belief space partition}, with a 3-dimensional belief simplex $\mathcal{B}$, and where $\pi^*(\cdot, \mathbf{b}) = \delta^*_j$ for all $\mathbf{b}$ in partition region $\mathcal{B}_j$.}
\end{figure}
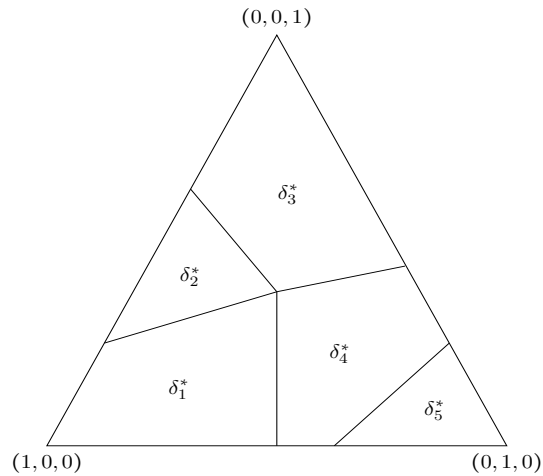

\begin{figure}[H]
	\hrulefill
	\centering
	\begin{enumerate}
		\item Generate a finite set of belief points, $\mathcal{B}' \subset \mathcal{B}$, via belief trajectory simulation (as in Appendix E). Let $\mathcal{B}' = \{\mathbf{b}_1, \ldots, \mathbf{b}_K \}$.
		
		\item For each $\mathbf{b} \in \mathcal{B}'$, generate $N$ demand trajectories $(y^n_1, \ldots, y^n_\tau)$. This gives us an estimate of the probabilities we need to compute the base stock level $a^*(\mathbf{b})$: 
		\[	\hat{P} \left[ \sum_{j=1}^\tau y_j  = \delta \vert \mathbf{b} \right] = \frac{1}{N} \sum_{n=1}^N \mathbf{1} \left\{ \sum_{j=1}^\tau y^n_j = \delta \right\}.	\]
		
		\item Calculate the estimated base stock level, $\hat{a}(\mathbf{b})$, for each $\mathbf{b} \in \mathcal{B}'$:	
		\begin{equation}
		\hat{a}(\mathbf{b}) \in \argmin_{\tilde{a}} \left\{ \sum_{\delta \in \Delta} \hat{P} \left[ \sum_{j=1}^\tau y_j  = \delta \vert \mathbf{b} \right] \left[  \tilde{h} \left(\tilde{a} - \sum_{j=1}^{\tau} y_j \right)^+ + \tilde{p} \left(\sum_{j=1}^{\tau} y_j  - \tilde{a}\right)^+ \right] \right\}. 
		\label{eq: base stock policy definition}
		\end{equation}
		
		\item Generate the separating hyperplanes by training a multi-class linear, soft-margin SVM on the set of tuples $\{ (\mathbf{b}_i, \hat{a}(\mathbf{b}_i)): i = 1, \ldots, K \}$.
	\end{enumerate}
	\caption{Partitioning the belief space, $\mathcal{B}$.}
	\hrulefill
	\label{fig: partition B}
\end{figure}

\begin{figure}[H]
	\centering
	\begin{subfigure}{.48\textwidth}
		\centering
		\tdplotsetmaincoords{70}{60}
		\begin{tikzpicture}[tdplot_main_coords, scale = 4.5]
		\def\laxis{1.1}
		\def\ltriangle{1}
		\def\ltick{.02}
		%%% axes
		\draw [->] (0,0,0) -- (\laxis,0,0) node [below] {$\mathbf{b}(u_{(1)})$};
		\draw [->] (0,0,0) -- (0,\laxis,0) node [right] {$\mathbf{b}(u_{(2)})$};
		\draw [->] (0,0,0) -- (0,0,\laxis) node [left] {$\mathbf{b}(u_{(3)})$};
		
		%%% finite grid, X'
		\foreach \Point in {
			(0.079, 0.38, 0.541),
			(0.26, 0.455, 0.285),
			(0.472, 0.129, 0.399),
			(0.748, 0.015, 0.237),
			(0.653, 0.001, 0.346),
			(0.015, 0.304, 0.681),
			(0.519, 0.105, 0.376),
			(0.24, 0.292, 0.468),
			(0.197, 0.29, 0.513),
			(0.487, 0.036, 0.477),
			(0.142, 0.168, 0.69),
			(0.052, 0.415, 0.533),
			(0.127, 0.417, 0.456),
			(0.289, 0.237, 0.474),
			(0.053, 0.344, 0.603),
			(0.123, 0.202, 0.675),
			(0.167, 0.371, 0.462),
			(0.383, 0.138, 0.479),
			(0.113, 0.159, 0.728),
			(0.302, 0.415, 0.283),
			(0.319, 0.217, 0.464),
			(0.332, 0.059, 0.609),
			(0.428, 0.175, 0.397),
			(0.39, 0.183, 0.427),
			(0.071, 0.469, 0.46),
			(0.329, 0.155, 0.516),
			(0.113, 0.47, 0.417),
			(0.3, 0.165, 0.535),
			(0.098, 0.572, 0.33),
			(0.483, 0.136, 0.381),
			(0.422, 0.131, 0.447),
			(0.459, 0.009, 0.532),
			(0.197, 0.019, 0.784),
			(0.05, 0.528, 0.422),
			(0.553, 0.106, 0.341),
			(0.284, 0.309, 0.407),
			(0.465, 0.215, 0.32),
			(0.116, 0.345, 0.539),
			(0.434, 0.008, 0.558),
			(0.223, 0.173, 0.604),
			(0.087, 0.589, 0.324),
			(0.358, 0.154, 0.488),
			(0.576, 0.029, 0.395),
			(0.374, 0.213, 0.413),
			(0.38, 0.271, 0.349),
			(0.376, 0.089, 0.535),
			(0.55, 0.097, 0.353),
			(0.016, 0.577, 0.407),
			(0.577, 0.044, 0.379),
			(0.063, 0.619, 0.318),
			(0.462, 0.175, 0.363),
			(0.102, 0.2, 0.698),
			(0.421, 0.01, 0.569),
			(0.101, 0.108, 0.791)
		}{
			\node [color = violet] at \Point {\textbullet};
		}
		
		\foreach \Point in {
			(0.345, 0.443, 0.212),
			(0.021, 0.737, 0.242),
			(0.289, 0.467, 0.244),
			(0.629, 0.206, 0.165),
			(0.732, 0.037, 0.231)
		}{
			\node [color = blue] at \Point {\textbullet};
		}
		
		\foreach \Point in {
			(0.692, 0.214, 0.094),
			(0.057, 0.767, 0.176),
			(0.238, 0.584, 0.178),
			(0.531, 0.254, 0.215),
			(0.062, 0.738, 0.2),
			(0.805, 0.011, 0.184),
			(0.722, 0.155, 0.123),
			(0.4, 0.503, 0.097),
			(0.426, 0.44, 0.134),
			(0.801, 0.024, 0.175),
			(0.149, 0.794, 0.057),
			(0.43, 0.415, 0.155),
			(0.161, 0.826, 0.013),
			(0.446, 0.414, 0.14),
			(0.177, 0.624, 0.199),
			(0.02, 0.959, 0.021),
			(0.116, 0.785, 0.099),
			(0.437, 0.526, 0.037),
			(0.222, 0.698, 0.08),
			(0.432, 0.465, 0.103),
			(0.284, 0.62, 0.096),
			(0.027, 0.856, 0.117),
			(0.385, 0.395, 0.22),
			(0.582, 0.281, 0.137),
			(0.127, 0.808, 0.065),
			(0.321, 0.653, 0.026)
		}{
			\node [color = green] at \Point {\textbullet};
		}
		
		\foreach \Point in {
			(0.661, 0.32, 0.019),
			(0.661, 0.333, 0.006),
			(0.57, 0.38, 0.05),
			(0.633, 0.336, 0.031),
			(0.729, 0.261, 0.01),
			(0.864, 0.018, 0.118),
			(0.411, 0.557, 0.032),
			(0.581, 0.33, 0.089),
			(0.847, 0.044, 0.109),
			(0.776, 0.136, 0.088),
			(0.532, 0.45, 0.018)
		}{
			\node [color = orange] at \Point {\textbullet};
		}
		
		\foreach \Point in {
			(0.976, 0.005, 0.019),
			(0.934, 0.028, 0.038),
			(0.941, 0.024, 0.035),
			(0.712, 0.285, 0.003)
		}{
			\node [color = red] at \Point {\textbullet};
		}
		
		%%% axes ticks
		\pgfmathtruncatemacro{\nticks}{floor(\laxis)-1}
		\begin{scope}[
		help lines,
		every node/.style={inner sep=1pt,text=black}
		]
		\foreach \coord in {0.5, 1} {
			\draw (\coord,\ltick,0) -- ++(0,-\ltick,0) -- ++(0,0,\ltick)
			node [pos=1,above left] {\coord};
			\draw (\ltick,\coord,0) -- ++(-\ltick,0,0) -- ++(0,0,\ltick)
			node [pos=1, above right] {\coord};
			\draw (\ltick,0,\coord) -- ++(-\ltick,0,0) -- ++(0,\ltick,0)
			node [at start,above right] {\coord};
		}
		\end{scope}
		%%% figure
		\filldraw [opacity=0.15,black] (\ltriangle,0,0) -- (0,\ltriangle,0)
		-- (0,0,\ltriangle) -- cycle;
		\end{tikzpicture}
		\subcaption{Points randomly generated on the belief simplex, $\mathcal{B}$, and labeled.}
		\label{fig: Random points in grid}
	\end{subfigure}
	\hfill
	\begin{subfigure}{.48\textwidth}
		\centering
		\tdplotsetmaincoords{70}{130}
		\begin{tikzpicture}[tdplot_main_coords, scale = 4.5]
		\def\laxis{1.1}
		\def\ltriangle{1}
		\def\ltick{.02}
		%%% axes
		\draw [->] (0,0,0) -- (\laxis,0,0) node [below] {$\mathbf{b}(u_{(1)})$};
		\draw [->] (0,0,0) -- (0,\laxis,0) node [right] {$\mathbf{b}(u_{(2)})$};
		\draw [->] (0,0,0) -- (0,0,\laxis) node [left] {$\mathbf{b}(u_{(3)})$};
		
		%%% finite grid, X'
		\foreach \Point in {
			(0.079, 0.38, 0.541),
			(0.26, 0.455, 0.285),
			(0.472, 0.129, 0.399),
			(0.748, 0.015, 0.237),
			(0.653, 0.001, 0.346),
			(0.015, 0.304, 0.681),
			(0.519, 0.105, 0.376),
			(0.24, 0.292, 0.468),
			(0.197, 0.29, 0.513),
			(0.487, 0.036, 0.477),
			(0.142, 0.168, 0.69),
			(0.052, 0.415, 0.533),
			(0.127, 0.417, 0.456),
			(0.289, 0.237, 0.474),
			(0.053, 0.344, 0.603),
			(0.123, 0.202, 0.675),
			(0.167, 0.371, 0.462),
			(0.383, 0.138, 0.479),
			(0.113, 0.159, 0.728),
			(0.302, 0.415, 0.283),
			(0.319, 0.217, 0.464),
			(0.332, 0.059, 0.609),
			(0.428, 0.175, 0.397),
			(0.39, 0.183, 0.427),
			(0.071, 0.469, 0.46),
			(0.329, 0.155, 0.516),
			(0.113, 0.47, 0.417),
			(0.3, 0.165, 0.535),
			(0.098, 0.572, 0.33),
			(0.483, 0.136, 0.381),
			(0.422, 0.131, 0.447),
			(0.459, 0.009, 0.532),
			(0.197, 0.019, 0.784),
			(0.05, 0.528, 0.422),
			(0.553, 0.106, 0.341),
			(0.284, 0.309, 0.407),
			(0.465, 0.215, 0.32),
			(0.116, 0.345, 0.539),
			(0.434, 0.008, 0.558),
			(0.223, 0.173, 0.604),
			(0.087, 0.589, 0.324),
			(0.358, 0.154, 0.488),
			(0.576, 0.029, 0.395),
			(0.374, 0.213, 0.413),
			(0.38, 0.271, 0.349),
			(0.376, 0.089, 0.535),
			(0.55, 0.097, 0.353),
			(0.016, 0.577, 0.407),
			(0.577, 0.044, 0.379),
			(0.063, 0.619, 0.318),
			(0.462, 0.175, 0.363),
			(0.102, 0.2, 0.698),
			(0.421, 0.01, 0.569),
			(0.101, 0.108, 0.791)
		}{
			\node [color = violet] at \Point {\textbullet};
		}
		
		\foreach \Point in {
			(0.345, 0.443, 0.212),
			(0.021, 0.737, 0.242),
			(0.289, 0.467, 0.244),
			(0.629, 0.206, 0.165),
			(0.732, 0.037, 0.231)
		}{
			\node [color = blue] at \Point {\textbullet};
		}
		
		\foreach \Point in {
			(0.692, 0.214, 0.094),
			(0.057, 0.767, 0.176),
			(0.238, 0.584, 0.178),
			(0.531, 0.254, 0.215),
			(0.062, 0.738, 0.2),
			(0.805, 0.011, 0.184),
			(0.722, 0.155, 0.123),
			(0.4, 0.503, 0.097),
			(0.426, 0.44, 0.134),
			(0.801, 0.024, 0.175),
			(0.149, 0.794, 0.057),
			(0.43, 0.415, 0.155),
			(0.161, 0.826, 0.013),
			(0.446, 0.414, 0.14),
			(0.177, 0.624, 0.199),
			(0.02, 0.959, 0.021),
			(0.116, 0.785, 0.099),
			(0.437, 0.526, 0.037),
			(0.222, 0.698, 0.08),
			(0.432, 0.465, 0.103),
			(0.284, 0.62, 0.096),
			(0.027, 0.856, 0.117),
			(0.385, 0.395, 0.22),
			(0.582, 0.281, 0.137),
			(0.127, 0.808, 0.065),
			(0.321, 0.653, 0.026)
		}{
			\node [color = green] at \Point {\textbullet};
		}
		
		\foreach \Point in {
			(0.661, 0.32, 0.019),
			(0.661, 0.333, 0.006),
			(0.57, 0.38, 0.05),
			(0.633, 0.336, 0.031),
			(0.729, 0.261, 0.01),
			(0.864, 0.018, 0.118),
			(0.411, 0.557, 0.032),
			(0.581, 0.33, 0.089),
			(0.847, 0.044, 0.109),
			(0.776, 0.136, 0.088),
			(0.532, 0.45, 0.018)
		}{
			\node [color = orange] at \Point {\textbullet};
		}
		
		\foreach \Point in {
			(0.976, 0.005, 0.019),
			(0.934, 0.028, 0.038),
			(0.941, 0.024, 0.035),
			(0.712, 0.285, 0.003)
		}{
			\node [color = red] at \Point {\textbullet};
		}
		
		%%% True partitioning hyperplanes
		%\draw [-, color = red] (0.913, 0, 0.087) -- (0.698, 0.302, 0) node [below] {};
		%\draw [-, color = red] (0.862, 0, 0.138) -- (0.268, 0.733, 0) node [below] {};
		%\draw [-, color = red] (0.799, 0, 0.201) -- (0, 0.806, 0.194) node [below] {};
		%\draw [-, color = red] (0.722, 0, 0.278) -- (0, 0.722, 0.278) node [below] {};
		
		%%% SVM partitioning hyperplanes
		%\draw [-, color = red] (0.888, 0, 0.112) -- (0.787, 0.213, 0) node [below] {};
		\draw [-, color = red, line width = 0.25mm] (0.977, 0, 0.023) -- (0.967, 0.033, 0) node [below] {};
		\draw [-, color = orange, line width = 0.25mm] (0.837, 0, 0.163) -- (0.431, 0.569, 0) node [below] {};
		\draw [-, color = green, line width = 0.25mm] (0.774, 0, 0.226) -- (0, 0.766, 0.234) node [below] {};
		\draw [-, color = blue, line width = 0.25mm] (0.739, 0, 0.261) -- (0, 0.72, 0.28) node [below] {};
		
		\filldraw [opacity=.15, red] (1,0,0) -- (0.977, 0, 0.023)
		-- (0.967, 0.033, 0) -- cycle;
		
		\filldraw [opacity=.15, orange] (0.977, 0, 0.023)
		-- (0.837, 0, 0.163)  -- (0.431, 0.569, 0) -- (0.967, 0.033, 0) -- cycle;
		
		\filldraw [opacity=.15, green] (0.837, 0, 0.163)
		--  (0.774, 0, 0.226) -- (0, 0.766, 0.234) -- (0, 1, 0) -- (0.431, 0.569, 0)  -- cycle;
		
		\filldraw [opacity=.15, blue] (0.774, 0, 0.226) 
		--  (0.739, 0, 0.261) -- (0, 0.72, 0.28)  -- (0, 0.766, 0.234)   -- cycle;
		
		\filldraw [opacity=.15, violet] (0.739, 0, 0.261)
		--  (0, 0, 1) -- (0, 0.72, 0.28)   -- cycle;
		
		%%% axes ticks
		\pgfmathtruncatemacro{\nticks}{floor(\laxis)-1}
		\begin{scope}[
		help lines,
		every node/.style={inner sep=1pt,text=black}
		]
		\foreach \coord in {0.5, 1} {
			\draw (\coord,\ltick,0) -- ++(0,-\ltick,0) -- ++(0,0,\ltick)
			node [pos=1,above left] {\coord};
			\draw (\ltick,\coord,0) -- ++(-\ltick,0,0) -- ++(0,0,\ltick)
			node [pos=1, above right] {\coord};
			\draw (\ltick,0,\coord) -- ++(-\ltick,0,0) -- ++(0,\ltick,0)
			node [at start,above right] {\coord};
		}
		\end{scope}
		%\addplot3[scatter, only marks] file{grid.txt};
		%%% figure
		\filldraw [opacity=0,black] (\ltriangle,0,0) -- (0,\ltriangle,0)
		-- (0,0,\ltriangle) -- cycle;
		\end{tikzpicture}
		\subcaption{The SVM-generated partition of $\mathcal{B}$ with $C = 10$.}
		\label{fig: SVM grid with C = 10}
	\end{subfigure}
	
	\begin{subfigure}[h]{0.48\textwidth}
		\centering
		\tdplotsetmaincoords{70}{130}
		\begin{tikzpicture}[tdplot_main_coords, scale = 4.5]
		\def\laxis{1.1}
		\def\ltriangle{1}
		\def\ltick{.02}
		%%% axes
		\draw [->] (0,0,0) -- (\laxis,0,0) node [below] {$\mathbf{b}(u_{(1)})$};
		\draw [->] (0,0,0) -- (0,\laxis,0) node [right] {$\mathbf{b}(u_{(2)})$};
		\draw [->] (0,0,0) -- (0,0,\laxis) node [left] {$\mathbf{b}(u_{(3)})$};
		
		%%% finite grid, X'
		\foreach \Point in {
			(0.079, 0.38, 0.541),
			(0.26, 0.455, 0.285),
			(0.472, 0.129, 0.399),
			(0.748, 0.015, 0.237),
			(0.653, 0.001, 0.346),
			(0.015, 0.304, 0.681),
			(0.519, 0.105, 0.376),
			(0.24, 0.292, 0.468),
			(0.197, 0.29, 0.513),
			(0.487, 0.036, 0.477),
			(0.142, 0.168, 0.69),
			(0.052, 0.415, 0.533),
			(0.127, 0.417, 0.456),
			(0.289, 0.237, 0.474),
			(0.053, 0.344, 0.603),
			(0.123, 0.202, 0.675),
			(0.167, 0.371, 0.462),
			(0.383, 0.138, 0.479),
			(0.113, 0.159, 0.728),
			(0.302, 0.415, 0.283),
			(0.319, 0.217, 0.464),
			(0.332, 0.059, 0.609),
			(0.428, 0.175, 0.397),
			(0.39, 0.183, 0.427),
			(0.071, 0.469, 0.46),
			(0.329, 0.155, 0.516),
			(0.113, 0.47, 0.417),
			(0.3, 0.165, 0.535),
			(0.098, 0.572, 0.33),
			(0.483, 0.136, 0.381),
			(0.422, 0.131, 0.447),
			(0.459, 0.009, 0.532),
			(0.197, 0.019, 0.784),
			(0.05, 0.528, 0.422),
			(0.553, 0.106, 0.341),
			(0.284, 0.309, 0.407),
			(0.465, 0.215, 0.32),
			(0.116, 0.345, 0.539),
			(0.434, 0.008, 0.558),
			(0.223, 0.173, 0.604),
			(0.087, 0.589, 0.324),
			(0.358, 0.154, 0.488),
			(0.576, 0.029, 0.395),
			(0.374, 0.213, 0.413),
			(0.38, 0.271, 0.349),
			(0.376, 0.089, 0.535),
			(0.55, 0.097, 0.353),
			(0.016, 0.577, 0.407),
			(0.577, 0.044, 0.379),
			(0.063, 0.619, 0.318),
			(0.462, 0.175, 0.363),
			(0.102, 0.2, 0.698),
			(0.421, 0.01, 0.569),
			(0.101, 0.108, 0.791)
		}{
			\node [color = violet] at \Point {\textbullet};
		}
		
		\foreach \Point in {
			(0.345, 0.443, 0.212),
			(0.021, 0.737, 0.242),
			(0.289, 0.467, 0.244),
			(0.629, 0.206, 0.165),
			(0.732, 0.037, 0.231)
		}{
			\node [color = blue] at \Point {\textbullet};
		}
		
		\foreach \Point in {
			(0.692, 0.214, 0.094),
			(0.057, 0.767, 0.176),
			(0.238, 0.584, 0.178),
			(0.531, 0.254, 0.215),
			(0.062, 0.738, 0.2),
			(0.805, 0.011, 0.184),
			(0.722, 0.155, 0.123),
			(0.4, 0.503, 0.097),
			(0.426, 0.44, 0.134),
			(0.801, 0.024, 0.175),
			(0.149, 0.794, 0.057),
			(0.43, 0.415, 0.155),
			(0.161, 0.826, 0.013),
			(0.446, 0.414, 0.14),
			(0.177, 0.624, 0.199),
			(0.02, 0.959, 0.021),
			(0.116, 0.785, 0.099),
			(0.437, 0.526, 0.037),
			(0.222, 0.698, 0.08),
			(0.432, 0.465, 0.103),
			(0.284, 0.62, 0.096),
			(0.027, 0.856, 0.117),
			(0.385, 0.395, 0.22),
			(0.582, 0.281, 0.137),
			(0.127, 0.808, 0.065),
			(0.321, 0.653, 0.026)
		}{
			\node [color = green] at \Point {\textbullet};
		}
		
		\foreach \Point in {
			(0.661, 0.32, 0.019),
			(0.661, 0.333, 0.006),
			(0.57, 0.38, 0.05),
			(0.633, 0.336, 0.031),
			(0.729, 0.261, 0.01),
			(0.864, 0.018, 0.118),
			(0.411, 0.557, 0.032),
			(0.581, 0.33, 0.089),
			(0.847, 0.044, 0.109),
			(0.776, 0.136, 0.088),
			(0.532, 0.45, 0.018)
		}{
			\node [color = orange] at \Point {\textbullet};
		}
		
		\foreach \Point in {
			(0.976, 0.005, 0.019),
			(0.934, 0.028, 0.038),
			(0.941, 0.024, 0.035),
			(0.712, 0.285, 0.003)
		}{
			\node [color = red] at \Point {\textbullet};
		}
		
		%%% True partitioning hyperplanes
		%\draw [-, color = red] (0.913, 0, 0.087) -- (0.698, 0.302, 0) node [below] {};
		%\draw [-, color = red] (0.862, 0, 0.138) -- (0.268, 0.733, 0) node [below] {};
		%\draw [-, color = red] (0.799, 0, 0.201) -- (0, 0.806, 0.194) node [below] {};
		%\draw [-, color = red] (0.722, 0, 0.278) -- (0, 0.722, 0.278) node [below] {};
		
		%%% SVM partitioning hyperplanes
		%\draw [-, color = red] (0.888, 0, 0.112) -- (0.787, 0.213, 0) node [below] {};
		\draw [-, color = red, line width = 0.25mm] (0.915, 0, 0.085) -- (0.81, 0.19, 0) node [below] {};
		\draw [-, color = orange, line width = 0.25mm] (0.837, 0, 0.163) -- (0.342, 0.658, 0) node [below] {};
		\draw [-, color = green, line width = 0.25mm] (0.824, 0, 0.176) -- (0, 0.745, 0.255) node [below] {};
		\draw [-, color = blue, line width = 0.25mm] (0.735, 0, 0.265) -- (0, 0.727, 0.273) node [below] {};
		
		\filldraw [opacity=.15, red] (1,0,0) -- (0.915, 0, 0.085)
		-- (0.81, 0.19, 0) -- cycle;
		
		\filldraw [opacity=.15, orange] (0.915, 0, 0.085)
		-- (0.837, 0, 0.163)  -- (0.342, 0.658, 0) -- (0.81, 0.19, 0)-- cycle;
		
		\filldraw [opacity=.15, green] (0.837, 0, 0.163)
		--  (0.824, 0, 0.176) --  (0, 0.745, 0.255) -- (0, 1, 0) -- (0.342, 0.658, 0)  -- cycle;
		
		\filldraw [opacity=.15, blue] (0.824, 0, 0.176)
		--  (0.735, 0, 0.265) -- (0, 0.727, 0.273)  --  (0, 0.745, 0.255)   -- cycle;
		
		\filldraw [opacity=.15, violet] (0.735, 0, 0.265)
		--  (0, 0, 1) -- (0, 0.727, 0.273)   -- cycle;
		
		%%% axes ticks
		\pgfmathtruncatemacro{\nticks}{floor(\laxis)-1}
		\begin{scope}[
		help lines,
		every node/.style={inner sep=1pt,text=black}
		]
		\foreach \coord in {0.5, 1} {
			\draw (\coord,\ltick,0) -- ++(0,-\ltick,0) -- ++(0,0,\ltick)
			node [pos=1, above left] {\coord};
			\draw (\ltick,\coord,0) -- ++(-\ltick,0,0) -- ++(0,0,\ltick)
			node [pos=1, above right] {\coord};
			\draw (\ltick,0,\coord) -- ++(-\ltick,0,0) -- ++(0,\ltick,0)
			node [at start, above right] {\coord};
		}
		\end{scope}
		%\addplot3[scatter, only marks] file{grid.txt};
		%%% figure
		\filldraw [opacity=0,black] (\ltriangle,0,0) -- (0,\ltriangle,0)
		-- (0,0,\ltriangle) -- cycle;
		\end{tikzpicture}
		\subcaption{The SVM-generated partition of $\mathcal{B}$ with $C = 50$.}
		\label{fig: SVM grid with C = 50}
	\end{subfigure}
	\hfill
	\begin{subfigure}[h]{0.48\linewidth}
		\centering
		\tdplotsetmaincoords{70}{130}
		\begin{tikzpicture}[tdplot_main_coords, scale = 4.5]
		\def\laxis{1.1}
		\def\ltriangle{1}
		\def\ltick{.02}
		%%% axes
		\draw [->] (0,0,0) -- (\laxis,0,0) node [below] {$\mathbf{b}(u_{(1)})$};
		\draw [->] (0,0,0) -- (0,\laxis,0) node [right] {$\mathbf{b}(u_{(2)})$};
		\draw [->] (0,0,0) -- (0,0,\laxis) node [left] {$\mathbf{b}(u_{(3)})$};
		
		%%% finite grid, X'
		\foreach \Point in {
			(0.079, 0.38, 0.541),
			(0.26, 0.455, 0.285),
			(0.472, 0.129, 0.399),
			(0.748, 0.015, 0.237),
			(0.653, 0.001, 0.346),
			(0.015, 0.304, 0.681),
			(0.519, 0.105, 0.376),
			(0.24, 0.292, 0.468),
			(0.197, 0.29, 0.513),
			(0.487, 0.036, 0.477),
			(0.142, 0.168, 0.69),
			(0.052, 0.415, 0.533),
			(0.127, 0.417, 0.456),
			(0.289, 0.237, 0.474),
			(0.053, 0.344, 0.603),
			(0.123, 0.202, 0.675),
			(0.167, 0.371, 0.462),
			(0.383, 0.138, 0.479),
			(0.113, 0.159, 0.728),
			(0.302, 0.415, 0.283),
			(0.319, 0.217, 0.464),
			(0.332, 0.059, 0.609),
			(0.428, 0.175, 0.397),
			(0.39, 0.183, 0.427),
			(0.071, 0.469, 0.46),
			(0.329, 0.155, 0.516),
			(0.113, 0.47, 0.417),
			(0.3, 0.165, 0.535),
			(0.098, 0.572, 0.33),
			(0.483, 0.136, 0.381),
			(0.422, 0.131, 0.447),
			(0.459, 0.009, 0.532),
			(0.197, 0.019, 0.784),
			(0.05, 0.528, 0.422),
			(0.553, 0.106, 0.341),
			(0.284, 0.309, 0.407),
			(0.465, 0.215, 0.32),
			(0.116, 0.345, 0.539),
			(0.434, 0.008, 0.558),
			(0.223, 0.173, 0.604),
			(0.087, 0.589, 0.324),
			(0.358, 0.154, 0.488),
			(0.576, 0.029, 0.395),
			(0.374, 0.213, 0.413),
			(0.38, 0.271, 0.349),
			(0.376, 0.089, 0.535),
			(0.55, 0.097, 0.353),
			(0.016, 0.577, 0.407),
			(0.577, 0.044, 0.379),
			(0.063, 0.619, 0.318),
			(0.462, 0.175, 0.363),
			(0.102, 0.2, 0.698),
			(0.421, 0.01, 0.569),
			(0.101, 0.108, 0.791)
		}{
			\node [color = violet] at \Point {\textbullet};
		}
		
		\foreach \Point in {
			(0.345, 0.443, 0.212),
			(0.021, 0.737, 0.242),
			(0.289, 0.467, 0.244),
			(0.629, 0.206, 0.165),
			(0.732, 0.037, 0.231)
		}{
			\node [color = blue] at \Point {\textbullet};
		}
		
		\foreach \Point in {
			(0.692, 0.214, 0.094),
			(0.057, 0.767, 0.176),
			(0.238, 0.584, 0.178),
			(0.531, 0.254, 0.215),
			(0.062, 0.738, 0.2),
			(0.805, 0.011, 0.184),
			(0.722, 0.155, 0.123),
			(0.4, 0.503, 0.097),
			(0.426, 0.44, 0.134),
			(0.801, 0.024, 0.175),
			(0.149, 0.794, 0.057),
			(0.43, 0.415, 0.155),
			(0.161, 0.826, 0.013),
			(0.446, 0.414, 0.14),
			(0.177, 0.624, 0.199),
			(0.02, 0.959, 0.021),
			(0.116, 0.785, 0.099),
			(0.437, 0.526, 0.037),
			(0.222, 0.698, 0.08),
			(0.432, 0.465, 0.103),
			(0.284, 0.62, 0.096),
			(0.027, 0.856, 0.117),
			(0.385, 0.395, 0.22),
			(0.582, 0.281, 0.137),
			(0.127, 0.808, 0.065),
			(0.321, 0.653, 0.026)
		}{
			\node [color = green] at \Point {\textbullet};
		}
		
		\foreach \Point in {
			(0.661, 0.32, 0.019),
			(0.661, 0.333, 0.006),
			(0.57, 0.38, 0.05),
			(0.633, 0.336, 0.031),
			(0.729, 0.261, 0.01),
			(0.864, 0.018, 0.118),
			(0.411, 0.557, 0.032),
			(0.581, 0.33, 0.089),
			(0.847, 0.044, 0.109),
			(0.776, 0.136, 0.088),
			(0.532, 0.45, 0.018)
		}{
			\node [color = orange] at \Point {\textbullet};
		}
		
		\foreach \Point in {
			(0.976, 0.005, 0.019),
			(0.934, 0.028, 0.038),
			(0.941, 0.024, 0.035),
			(0.712, 0.285, 0.003)
		}{
			\node [color = red] at \Point {\textbullet};
		}
		
		%%% True partitioning hyperplanes
		\draw [-, color = red, line width = 0.25mm] (0.913, 0, 0.087) -- (0.698, 0.302, 0) node [below] {};
		\draw [-, color = orange, line width = 0.25mm] (0.862, 0, 0.138) -- (0.268, 0.733, 0) node [below] {};
		\draw [-, color = green, line width = 0.25mm] (0.799, 0, 0.201) -- (0, 0.806, 0.194) node [below] {};
		\draw [-, color = blue, line width = 0.25mm] (0.722, 0, 0.278) -- (0, 0.722, 0.278) node [below] {};
		
		%%% SVM partitioning hyperplanes
		%\draw [-, color = red] (0.888, 0, 0.112) -- (0.787, 0.213, 0) node [below] {};
		%\draw [-, color = red, line width = 0.25mm] (0.977, 0, 0.023) -- (0.967, 0.033, 0) node [below] {};
		%\draw [-, color = orange, line width = 0.25mm] (0.837, 0, 0.163) -- (0.431, 0.569, 0) node [below] {};
		%\draw [-, color = green, line width = 0.25mm] (0.774, 0, 0.226) -- (0, 0.766, 0.234) node [below] {};
		%\draw [-, color = blue, line width = 0.25mm] (0.739, 0, 0.261) -- (0, 0.72, 0.28) node [below] {};
		
		\filldraw [opacity=.15, red] (1,0,0) -- (0.913, 0, 0.087)
		-- (0.698, 0.302, 0) -- cycle;
		
		\filldraw [opacity=.15, orange] (0.913, 0, 0.087)
		-- (0.862, 0, 0.138)  -- (0.268, 0.733, 0) -- (0.698, 0.302, 0) -- cycle;
		
		\filldraw [opacity=.15, green] (0.862, 0, 0.138)
		--  (0.799, 0, 0.201) -- (0, 0.806, 0.194) -- (0, 1, 0) -- (0.268, 0.733, 0)  -- cycle;
		
		\filldraw [opacity=.15, blue] (0.799, 0, 0.201)
		--  (0.722, 0, 0.278) -- (0, 0.722, 0.278)  -- (0, 0.806, 0.194)   -- cycle;
		
		\filldraw [opacity=.15, violet] (0.722, 0, 0.278)
		--  (0, 0, 1) -- (0, 0.722, 0.278)   -- cycle;
		
		%%% axes ticks
		\pgfmathtruncatemacro{\nticks}{floor(\laxis)-1}
		\begin{scope}[
		help lines,
		every node/.style={inner sep=1pt,text=black}
		]
		\foreach \coord in {0.5, 1} {
			\draw (\coord,\ltick,0) -- ++(0,-\ltick,0) -- ++(0,0,\ltick)
			node [pos=1, above left] {\coord};
			\draw (\ltick,\coord,0) -- ++(-\ltick,0,0) -- ++(0,0,\ltick)
			node [pos=1, above right] {\coord};
			\draw (\ltick,0,\coord) -- ++(-\ltick,0,0) -- ++(0,\ltick,0)
			node [at start, above right] {\coord};
		}
		\end{scope}
		%\addplot3[scatter, only marks] file{grid.txt};
		%%% figure
		\filldraw [opacity=0,black] (\ltriangle,0,0) -- (0,\ltriangle,0)
		-- (0,0,\ltriangle) -- cycle;
		\end{tikzpicture}
		\caption{The true partition of $\mathcal{B}$.}
		\label{fig: true grid}
	\end{subfigure}
	
	%\captionsetup{width=0.8\textwidth}
	\caption[LoF entry]{Depicting example SVM partitions of $\mathcal{B}$ under different values of the SVM regularization parameter, $C$. The regions correspond to different values of the optimal base stock levels. In this example, the HMM latent state space has three elements $\mathcal{U} = \{u_{(1)}, u_{(2)}, u_{(3)}\}$, the AOD space has three elements $\mathcal{X} = \{x_{(1)}, x_{(2)}, x_{(3)}\}$, and the demand space has five elements $\mathcal{Y} = \{1, 2, 3, 4, 5\}$. The dynamics $P[y', x', u' \vert u]$ are governed by three matrices $U$, $Q$, and $Y$: 
		\begin{equation*}
		U = 
		\begin{bmatrix}
		0.75 & 0.125 & 0.125 \\
		0.125 & 0.75 & 0.125 \\
		0.125 & 0.125 & 0.75 
		\end{bmatrix}, \quad
		Q = 
		\begin{bmatrix}
		0.9 & 0.05 & 0.05 \\
		0.05 & 0.9 & 0.05 \\
		0.05 & 0.05 & 0.9 
		\end{bmatrix}, \quad
		Y = \begin{bmatrix}
		0.75 & 0.1 & 0.05 & 0.05 & 0.05 \\
		0.05 & 0.075 & 0.75 & 0.075 & 0.05 \\
		0.05 & 0.05 & 0.05  & 0.1 & 0.75
		\end{bmatrix},
		\end{equation*}
	where $U(i, j) = P[u' = u_{(j)} \vert u = u_{(i)}]$, $Q(i, k) = P[x' = x_{(k)} \vert u = u_{(i)}]$, $Y(i, l) = P[y' = l \vert u = u_{(i)}]$, and $P[y' = l, x' = x_{(k)}, u' = u_{(j)} \vert u = u_{(i)}] = U(i, j) Q(i, k) Y(i, l)$. The lead time is $\tau = 2$, the discount factor $\beta = 0.9$, $\tilde{p} = 70$, $\tilde{h} = 10$. }
	\label{fig: all partition grids}
	\hrulefill
\end{figure}
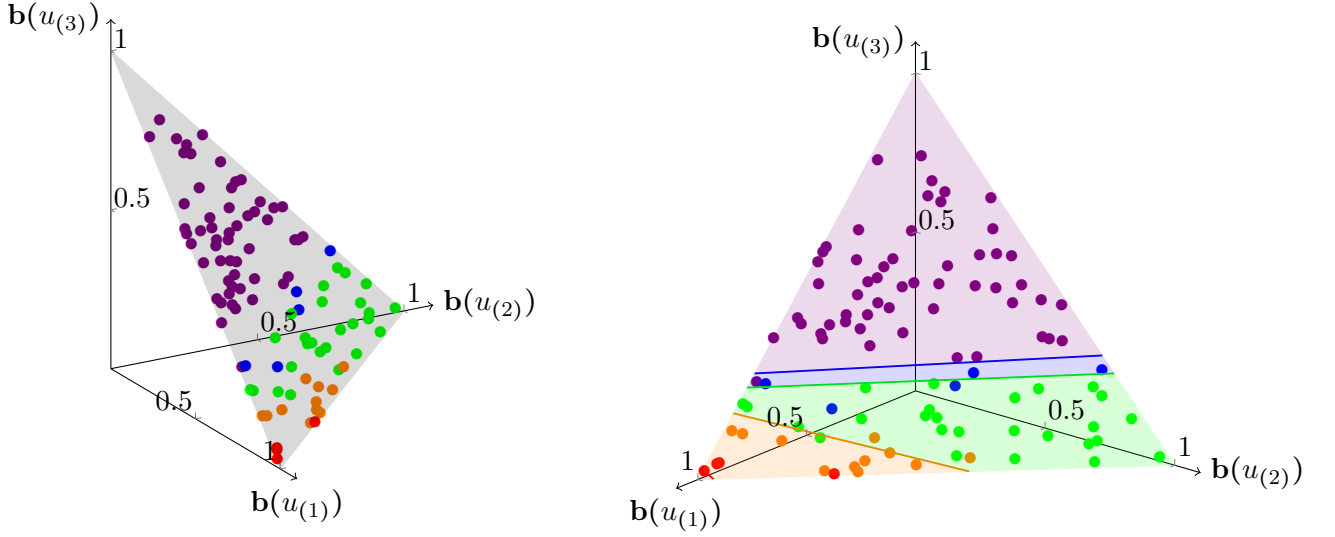
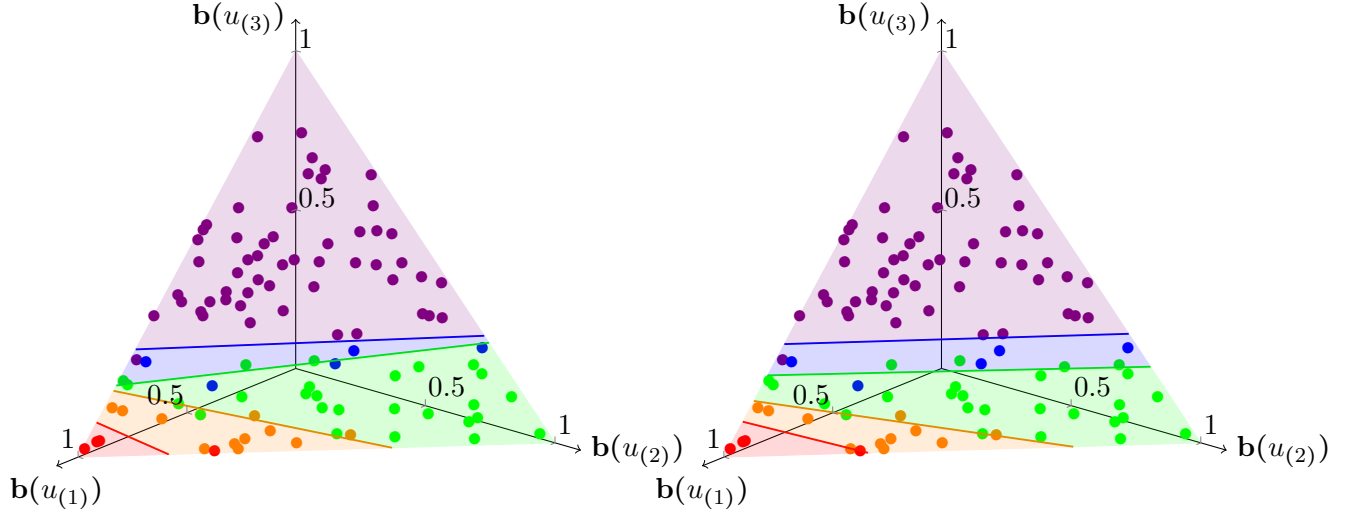

\begin{figure}[h]
	\hrulefill
	
	\textbf{Evaluate}($\mathcal{H}^\text{true}$, $\mathcal{H}^\text{eval}$, $\mathrm{SVM}^{\text{eval}}$, $\beta$, $\tau$, $\tilde{p}$, $\tilde{h}$, $T$, $N^{\text{sim}}$): For each Monte Carlo simulation $n = 1, \ldots, N^{\text{sim}}$, generate $v^n$ as follows.
	
	\begin{enumerate}
		\item Initialize $s^n_0 = 0$, $x^n_0 = \left[ \frac{1}{3}, \frac{1}{3}, \frac{1}{3} \right]$, $d^n_0 = 0$, $a^n_{-1} = \ldots = a^n_{-\tau} = 0$, and $v^n_\theta = 0$. Sample $u_0$ from the belief distribution $x_0$.
		
		\item For $t = 0, \ldots, T$:
		\begin{itemize}
			\item \textit{Determine ordering decision and cost.}
			\begin{flalign*}
			\tilde{a}^n_{t} &\gets \mathrm{SVM}^{\text{eval}}(\mathbf{b}^n_t)	&\\
			\tilde{s}^n_t 	&\gets s^n_t - \sum_{j = 1}^{\tau}{a^n_{t-j}} - d^n_t &\\
			a^n_t	&\gets \left( \tilde{a}^n_{t}- \tilde{s}^n_t  \right)^+ &\\
			v^n	&\gets v^n + \beta^t \big[ \tilde{h} \left( s^n_t + a^n_{t-\tau} - y^n_t \right)^+ + \tilde{p} \left( y^n_t - s^n_t - a^n_{t-\tau} \right)^+ \big]&
			\end{flalign*}
			\item \textit{Transition, costs, and belief update.}
			\begin{flalign*}
			&s^n_{t+1} \gets s^n_t + a^n_{t-\tau} - y^n_t	&\\
			&(y^n_{t+1}, x^n_{t+1}, u^n_{t+1}) \sim \mathcal{H}^\text{true}	& \\
			&\mathbf{b}^n_{t+1} \gets \lambda_{\mathcal{H}^\text{eval}}(y^n_{t+1}, x^n_{t+1}, \mathbf{b}^n_t) &
			\end{flalign*}
		\end{itemize}
	\end{enumerate}
	\textbf{Return:} $\sum_{n=1}^{N^{\text{sim}}} \frac{v^n}{N^{\text{sim}}}$
	\caption{The SVM-Monte Carlo policy evaluation method.}
	\label{fig: SVM-Monte Carlo evaluation method}
	\hrulefill
\end{figure}

\begin{figure}[H]
	\centering

	\includegraphics[width=0.8\linewidth]{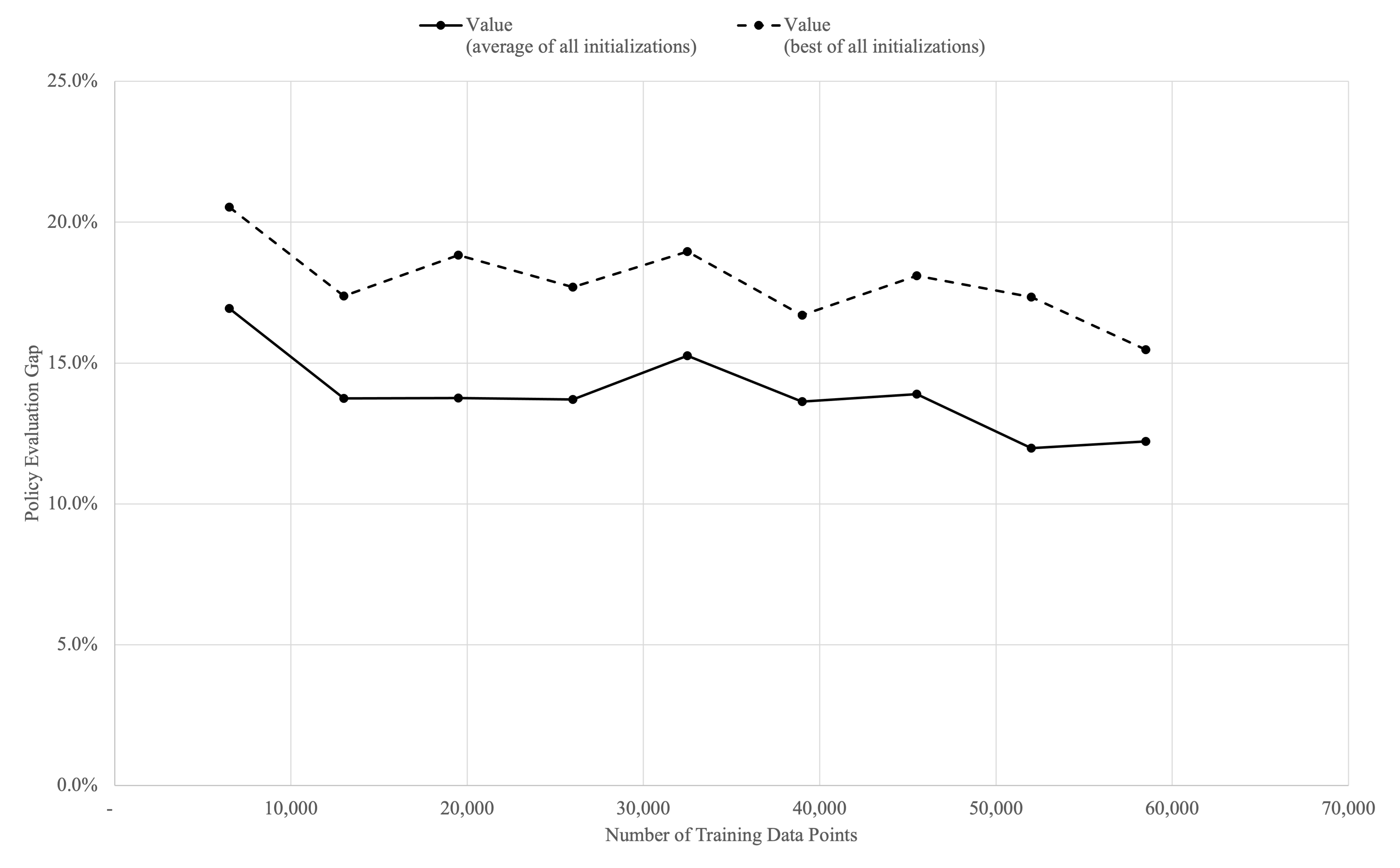}
	\caption{The optimality gap between the SVM-generated base stock policy under the true HMM demand model and under HMMs trained on synthetic data.}
	\label{fig: computational example graph}
	\hrulefill
\end{figure}

\newpage
\begin{appendices}

	% NEW APPENDIX SECTION
	\section{Relationship to Statistics and Machine Learning}
	\subsection{Markov Forecasting Models}
	\textbf{\emph{Brownian motion.}} We might consider modeling the auxiliary process (which we briefly assume to be univariate), $\{x_t: t \geq 0\}$, as standard Brownian motion, which satisfies the following two properties (\cite{Resnick92}, chapter 6): (1) $\{x_t: t \geq 0\}$ has independent increments and (2) $x_{t+1} - x_t \sim \mathcal{N}(0,1)$. In our notation, the conditional probability distribution is:
		\begin{equation*}
			\begin{split}
				P_{\mathscr{D}}[x', u', \theta' \vert x, u, \theta] &= P_{\mathscr{D}}[x' \vert x] \\
				&= \frac{1}{\sqrt{2 \pi}} e^{-\frac{1}{2} \left( x - x' \right)^2}.
			\end{split}
		\end{equation*}
	Standard Brownian motion satisfies the Markov assumption, by virtue of its independent increments, as do other examples of L\'evy processes. In fact, the Markov assumption holds for other generalizations built upon standard Brownian motion that are popular particularly in mathematical finance (we will discuss one such application from \cite{Zhou09}, later), for example, Brownian motion with drift, geometric Brownian motion (popularized by its use as a model of the underlying stock price process in the Black-Scholes model), and Ornstein-Uhlenbeck processes (\cite{Resnick92}, \cite{Zhou09}). For standard Brownian motion and these generalizations, there exist Markovian extensions in the case of a vector-valued auxiliary process, enabling modeling flexibility with correlated auxiliary data.
		
	\textbf{\emph{Autoregressive time series models.}} Autoregressive time series models are some of the more popular forecasting models used in practice for time series with regular and discrete time intervals (\cite{hyndman2008automatic}). For example, the modeler might assume that the auxiliary process is an ``autoregressive moving average" process, with parameters $p$ and $q$ (call this ARMA$(p,q)$) determining that for all $t$, $x_t$ is dependent upon the past $p$ realizations of the auxiliary data process, $x_{t-1}, \ldots, x_{t-p}$, and the average of the previous $q$ realizations of the noise process $\{ u_{t}: t \geq 0\}$:
		\begin{equation*}
			x_t = \theta^{\text{intercept}} + u_t + \sum_{j=1}^p \theta^{\text{AR}}_{j} x_{t-j} + \sum_{j=1}^q \theta^{\text{MA}}_j u_{t-j}, \quad u_{t} \stackrel{\text{i.i.d.}}{\sim} \mathcal{N}(0, \sigma_{u}^2).
		\end{equation*}
	Note that $x_t$ is Markovian with respect to the vector $(x_{t-1}, \ldots, x_{t - p}, u_{t-1}, \ldots, u_{t-q})$. Let $\tilde{x}_{t} = (x_{t}, \ldots, x_{t - p+1})$ be the previous $p$ observations of the auxiliary process before time $t$ and let $\tilde{u}_t = (u_{t}, \ldots, u_{t - q+1})$ be the previous $q$ observations of the $u$-process. Since the $u$-process is assumed to be i.i.d., $(\tilde{x}_t, \tilde{u}_t)$ satisfies the forecasting conditional probability of (\ref{eq: discriminative learning, plus forecasting conditional probabilities}) with fixed parameters $\theta = (\theta^{\text{intercept}}, \theta^{AR}_1, \ldots, \theta^{AR}_p, \theta^{MA}_1, \ldots, \theta^{MA}_q, \sigma_u)$:
		\begin{equation*}
			\begin{split}
				P_{\mathscr{D}}&[\tilde{x}_{t+1}, \tilde{u}_{t+1}, \theta_{t+1} \vert \tilde{x}_t, \tilde{u}_t, \theta_t] \\
				&= P_{\mathscr{D}}[x_{t+1}, \ldots, x_{t-p+2}, u_{t+1}, \ldots, u_{t-q+2} \vert x_{t}, \ldots, x_{t-p+1}, u_{t}, \ldots, u_{t-q+1}, \theta] \\
				&= P_{\mathscr{D}}[x_{t+1} = x' \vert u_{t+1}, x_{t}, \ldots, x_{t-p+1}, u_{t}, \ldots, u_{t-q+1}, \theta] P_{\mathscr{D}}[u_{t+1} \vert \theta],
			\end{split}
		\end{equation*}
	and the following marginal conditional distribution over $x_{t+1}$ is normally-distributed:
		\begin{equation*}
			\begin{split}
				P_{\mathscr{D}}[x_{t+1} &= x' \vert x_{t}, \ldots, x_{t-p+1}, u_{t}, \ldots, u_{t-q+1}, \theta] \\
				&= \frac{1}{\sigma_u \sqrt{2 \pi}} \exp \left\{ -\frac{1}{2} \left( \frac{x' - \theta^{\text{intercept}} - \sum_{j=1}^p \theta^{\text{AR}}_{j} x_{t-j} - \sum_{j=1}^q \theta^{\text{MA}}_j u_{t-j}}{\sigma_u} \right)^2 \right\}.
			\end{split}
		\end{equation*}
		
	This kind of autoregressive model can be extended to the case of a vector-valued auxiliary process in the \emph{vector autoregressive} (VAR) model (\cite{watson1994vector}).

	% NEW APPENDIX SECTION
	\section{Monotone Approximate Dynamic Programming}
	Consider the finite horizon MDP model considered in \cite{JiangPowell15}, which we succinctly describe by the following Bellman equation (and without loss of generality, we assume a \emph{minimization} formulation, to more easily facilitate comparison to our SEP-POMDP framework):
		\begin{equation*}
			\begin{split}
				v^*_t(s) &= \min_{a \in \mathcal{A}} \big\{ c_t(s, a) + \mathbb{E}\left[ v^*_{t+1} ( s_{t+1} ) \vert s_t = s, a_t = a \right] \big\}, \quad t = 0, 1, 2, \ldots, T-1 \\
				v_{T}(s) &= c_T(s),
			\end{split}
		\end{equation*}
	where the state transition dynamics are described by the stochastic function, $s_{t+1} = f(s_t, a_t, w_{t+1})$ and $\{w_t: t \geq 0\}$ is a stochastic process (which \cite{JiangPowell15} call the ``information process"), in a space $\mathscr{W}$, meant to capture the totality of the stochasticity in state dynamics. Now, we note that this MDP formulation corresponds to the MDP analogs of (\ref{eq: optimality equation of MDP analog}) (albeit with a description of state dynamics via a stochastic function, rather than the equivalent conditional probability specification). The only difference is the introduction of an affixed value of the observation process, $y'$:
		\begin{equation}
			\begin{split}
				v^*_t(s) &= \min_{a \in \mathcal{A}} \big\{ c_t(y', s, a) + \mathbb{E}\left[ v^*_{t+1} ( s_{t+1} ) \vert y', s_t = s, a_t = a \right] \big\}, \quad t = 0, 1, 2, \ldots, T-1 \\
				v_{T}(s) &= c_T(s),
			\end{split}
			\label{eq: JiangPowell MDP analog formulation}
		\end{equation}
		where $s_{t+1} = f(s_t, a_t, y', w_{t+1})$. We might consider $y'$ as introducing an observed component of the information process, which in the MDP analog formulation is affixed, but for the SEP-POMDP we permit to be a random variable that is useful for explaining (at least part) of the uncertainty captured by the information process, and for which we want to build a statistical learning model for describing, as in Section 5.
	
	\cite{JiangPowell15} are principally focused on MDPs for which the optimal value functions exhibit the following monotonicity property, for all $t$:
		\begin{equation}
			s \preceq \tilde{s} \Rightarrow v^*_t(s) \geq v^*_{t}(\tilde{s}), \quad \forall t = 0, 1, 2, \ldots, T \text{ and } s, \tilde{s} \in \mathcal{S},
			\label{eq: JiangPowell monotonicity property}
		\end{equation}
	where $\preceq$ is a component-wise partial order, such that when the state can be decomposed into $s = (m, j)$ (where $m$ is in a space $\mathscr{M}$ and $j$ in a space $\mathscr{J}$):
		\begin{equation*}
			s \preceq \tilde{s} \Leftrightarrow m \leq \tilde{m}, j = \tilde{j}.
		\end{equation*}
		
	They present a proposition with sufficient conditions under which the optimal value functions exhibit the monotonicity property (\ref{eq: JiangPowell monotonicity property}), that we include verbatim, below, with only trivial modifications to facilitate comparison to our MDP analog formulation (\ref{eq: JiangPowell MDP analog formulation}). We then demonstrate that the assumptions of this proposition guaranteeing monotone value functions for the MDPs in \cite{JiangPowell15} satisfies the conditions of Corollary \ref{cor: value function structure}, and thus the SEP-POMDPs that include statistical learning models for explaining the $y$-process (an observed component of the information process of \cite{JiangPowell15}) inherit this monotone value function structure.
	
	\begin{proposition}[\cite{JiangPowell15}, Proposition 1]	
		Suppose that every $s \in \mathcal{S}$ can be written as $s = (m, i)$ for some $m \in \mathscr{M}$ and $j \in \mathscr{J}$, and let $s_t = (m_t, j_t)$, be the state at time $t$, with $m_t \in \mathscr{M}$ and $j_t \in \mathscr{J}$. Assume:
			\begin{enumerate}[label = JP\arabic*.]
				\item For every $s, \tilde{s} \in \mathcal{S}$ with $s \preceq \tilde{s}$, $a \in \mathcal{A}$, and $w \in \mathscr{W}$, the state transition function satisfies $f(s, a, y', w) \preceq f(\tilde{s}, a, y', w)$,
				
				\item For each $t < T$, $s, \tilde{s} \in \mathcal{S}$, with $s \preceq \tilde{s}$, and $a \in \mathcal{A}$, $c_t(s, a) \geq c(\tilde{s}, a)$ and $c_T(s) \geq c_T(\tilde{s})$.
				
				\item For each $t < T$, $m_t$ and $w_{t+1}$ are independent.
			\end{enumerate}
		Then the value functions $v^*_t$ satisfy the monotonicity property (\ref{eq: JiangPowell monotonicity property}).
	\end{proposition}
	
	We will prove the following inheritance proposition, proving SEP-POMDP inheritance of monotone optimal value function structure under conditions JP1-JP3.
	
	\begin{proposition}[SEP-POMDP inheritance under \cite{JiangPowell15} monotonicity conditions.]
		Suppose JP1, JP2, and JP3 hold. Then, for the SEP-POMDP $v^*_t$, for $t = 0, 1, 2 \ldots, T$, satisfies the monotonicity property for all $\mathbf{b} \in \mathcal{B}$. That is, for $s, \tilde{s} \in \mathcal{S}$:
			\begin{equation*}
				s \preceq \tilde{s} \Rightarrow v^*_t(s, \mathbf{b}) \geq v^*_{t}(\tilde{s}, \mathbf{b}), \quad \forall t = 0, 1, 2, \ldots, T, \forall \mathbf{b} \in \mathcal{B}.
			\end{equation*}
		\label{prop: JiangPowell SEP-POMDP inheritance}
	\end{proposition}

	\begin{proof}[Proof of Proposition \ref{prop: JiangPowell SEP-POMDP inheritance}]
		It suffices to show that JP1-JP3 imply P(a), B(a), and B(b). We begin by explicitly defining the structured functional spaces (implicit in \cite{JiangPowell15}):
			\begin{equation*}
			\begin{split}
				\tilde{V} &\triangleq \{v: \mathcal{S} \mapsto \mathbb{R}: v \text{ satisfies the monotonicity property (\ref{eq: JiangPowell monotonicity property})} \} \\
				\tilde{F} &\triangleq \{f: \mathcal{S} \times \mathcal{A} \mapsto \mathbb{R}: f(\cdot, a) \text{ satisfies the monotonicity property (\ref{eq: JiangPowell monotonicity property}) for all } a \in \mathcal{A} \}.
			\end{split}
			\end{equation*}
		The space of real-valued monotone functions is closed, so P(a) is satisfied. To show that B(b) holds, we utilize the results of \cite{Smith02} and note that the functions in $\tilde{F}$ satisfy a special kind of joint extension of a C3 property, $\mathscr{P}^*$, called \emph{single-point properties} (for minimization problems, \emph{e.g.} monotonicity, concavity in $\mathcal{S}$). Since the monotonicity property defining $\tilde{F}$ is a single-point property, it follows from \cite{Smith02} Proposition 4 that it is preserved under minimization. Finally, B(a) is satisfied by the following inductive argument (from \cite{JiangPowell15}) following from JP1 and JP3. Suppose $v^*_{t+1} \in \tilde{V}$ and $s, \tilde{s} \in \mathcal{S}$ such that $s \preceq \tilde{s}$:
			\begin{equation*}
				\begin{split}
					\mathbb{E} \left[ v^*_{t+1} f(s, a, y') \vert s_t = s, a_t = a, y' \right] &= \mathbb{E} \left[ f(s, a, y') \vert j_t = j, a_t = a, y' \right] \\
					&\geq \mathbb{E} \left[ f(\tilde{s}, a, y') \vert j_t = \tilde{j}, a_t = a, y' \right] \\
					&\geq \mathbb{E} \left[ f(\tilde{s}, a, y') \vert s_t = \tilde{s}, a_t = a, y' \right].
				\end{split}
			\end{equation*}
		Hence, $\mathbb{E} \left[ v^*_{t+1} \big( s_{t+1} \big) \vert \cdot, \cdot, y' \right] \in \tilde{F}$. By JP2, $c_t(\cdot, \cdot, y') \in \tilde{F}$. B(a) follows because $\tilde{F}$ is a convex cone.
	\end{proof}

	% NEW APPENDIX SECTION
	\section{Computational Example}
	\subsection{Proof of base stock optimality.}
	We will prove Proposition \ref{prop: base stock optimality} --- the optimality of a base stock policy for the single product inventory replenishment problem under procurement delays --- by showing how, since the problem can be formulated as a SEP-POMDP, it inherits this structure from an MDP analog. Rather than showing this directly, considering the context of an MDP analog inventory problem, we instead show conditions for SEP-POMDPs inheriting a more general myopic optimal policy structure from the MDPs considered in Theorem 1 of \cite{Sobel81}, and then demonstrate that our computational example satisfies the conditions for this myopic optimal policy structure.
	
	We will make use of the notion of \emph{separable} functions.
	\begin{definition} (separable function)
		A function $f: \mathcal{S} \times \mathcal{A} \mapsto \mathbb{R}$ is separable if there exists a function $K: \mathcal{A} \mapsto \mathbb{R}$ and a function $L: \mathcal{S} \mapsto \mathbb{R}$ such that $f(s,a) = L(s) + K(a) $.
	\end{definition}
	Note that the space of separable functions is a convex cone. That is, suppose we have two separable functions, $f$ and $g$, which map $\mathcal{S} \times \mathcal{A}$ to $\mathbb{R}$, and conic weights $\alpha, \beta \geq 0$. Clearly,
		\begin{equation*}
			\alpha f(s,a) + \beta g(s, a) = \alpha K_f(a) + \beta K_g(a) + \alpha L_f(s) + \beta L_g(s).
		\end{equation*}
	
	Now, we prove conditions for the optimality of myopic policies for the SEP-POMDP, that are inherited from the MDPs of \cite{Sobel81}, and we assume the spaces $\mathcal{S}$, $\mathcal{Y}$, $\mathcal{M}$, and $\mathcal{A}$ are discrete.
	\begin{proposition}[Myopic optimal polices.] Suppose the following:
		\begin{enumerate}[label = (\roman*)]
			\item $\exists K: \mathcal{A} \times \mathcal{Y} \mapsto \mathbb{R}, L: \mathcal{S} \times \mathcal{Y} \mapsto \mathbb{R}$ such that $c(s, y', a) = K(y', a) + L(s, y')$, for all $y' \in \mathcal{Y}$, $s \in \mathcal{S}$, $a \in \mathcal{A}$
			
			\item $p(\cdot \vert y', s, a)$ is independent of $s$ (and so we express as $p(\cdot \vert y', a)$), for all $y' \in \mathcal{Y}$, $a \in \mathcal{A}$
			
			\item $a^*(\mathbf{b})	\in 	\argmin\limits_{a \in \mathcal{A}} \{ G(\mathbf{b},a )\}$, where
			\begin{equation*}
				G(\mathbf{b}, a)		=		\sum_{y'} \sigma(y' \vert \mathbf{b}) \left[ K(y', a) + \beta \sum_{y''} \sigma \big(y'' \vert \lambda(y',\mathbf{b}) \big) \sum_{s'} 
			p(s' \vert y', a) L(s', y'') \right]
			\end{equation*}
			
			\item $a^*(\mathbf{b}_t)$ is feasible for all $t$
		\end{enumerate}
		
		Then, the stationary deterministic policy $\pi^*(s, \mathbf{b}) = a^*(\mathbf{b})$ for all $s \in \mathcal{S}, \mathbf{b} \in \mathcal{B}$ is optimal.
		\label{Myopic}
	\end{proposition}

	\begin{proof}[Proof of Proposition \ref{Myopic}]
		Suppose $v(\cdot, \mathbf{b}) \in \tilde{V}$ for all $\mathbf{b} \in \mathcal{B}$. We begin by defining the following structured function spaces:
		\begin{align*}
		\tilde{\Pi}	&\triangleq		\{	\tilde{\pi}		:	\exists a \in \mathcal{A}: 	\tilde{\pi}(s) = a,	\forall s \in \mathcal{S}	\}		\\
		\tilde{V} 	&\triangleq		V	\\
		\tilde{C} 	&\triangleq		\{	\tilde{c}			:	\exists K: \mathcal{A} \mapsto \mathbb{R}, L: \mathcal{S} \mapsto \mathbb{R}:		\tilde{c}(s,a) = K(a) + L(s)	\}				\\
		\tilde{P} 	&\triangleq		\{	\tilde{p}			:	\tilde{p}(\cdot \vert s, a) = \tilde{p}(\cdot \vert a)	\}				\\
		\tilde{F}	&\triangleq 	\{	f			:	\exists K: \mathcal{A} \mapsto \mathbb{R}, L: \mathcal{S} \mapsto \mathbb{R}:			{f}(s,a) = K(a) + L(s)	\}.			
		\end{align*}
		We want to show that there exists a set $\{ a(\mathbf{b}): \mathbf{b} \in \mathcal{B}\}$ such that $\pi^*(s, \mathbf{b}) = a(\mathbf{b})$ for all $(s, \mathbf{b}) \in \mathcal{S} \times \mathcal{B}$ is stationary optimal by showing that P(a), B(a), B(b), and B(c) hold. 
		
		P(a) holds trivially. We aim to show B(a) holds. Suppose $\tilde{v} \in \tilde{V}$. Observe that (i) and (ii) are equivalent to $p(\cdot \vert y', \cdot, \cdot) \in \tilde{P}$ for all $y' \in \mathcal{Y}$ and $c(\cdot, y', \cdot) \in \tilde{C}$ for all $y' \in \mathcal{Y}$, which imply that
		\begin{align*}
		h_{y'}(s, a, \tilde{v})		&=	c(s, y', a)	+	\beta \sum_{s' \in \mathcal{S}} p(s' \vert y', s, a) \tilde{v}(s')		\\
		&=	K(y', a) + L(s, y')	+ 	\beta \sum_{s' \in \mathcal{S}} p(s' \vert y', a) L(s') \in \tilde{F}, \text{ for all } y' \in \mathcal{Y}.
		\end{align*}
		B(b) trivially holds. Further, separable functions when minimized yield state-invariant optimal policies (maximizing $L(s) + K(a)$ over $a$ is equivalent to minimizing $K(a)$ over $a$ for all $s$). So B(c) holds. By Proposition \ref{prop: SEP-POMDP inheritance} we conclude that there exists a set $\{ a(\mathbf{b}): \mathbf{b} \in \mathcal{B}\}$ such that $\pi^*(s, \mathbf{b}) = a(\mathbf{b})$ for all $(s, \mathbf{b}) \in \mathcal{S} \times \mathcal{B}$ is stationary optimal.
		
		It remains to show that $\pi^*(s, \mathbf{b}) = a^*(\mathbf{b})$ for all $s \in \mathcal{S}$, the myopic minimizer of the function $G(\mathbf{b}, a)$. An inductive argument, which follows along the lines of the proof given in \cite{Sobel81} proves this result.
		
		Let $L(s, \mathbf{b}) = \mathbb{E}[ L(s, y') \vert \mathbf{b} ]$ and $K(\mathbf{b}, a) = \mathbb{E}[ K(y', a) \vert \mathbf{b}]$. The value function of the SEP-POMDP, under any policy $\pi$ is defined as follows, where $\mathbf{b}_{t+1} = \lambda(z_{t+1}, \mathbf{b}_t)$ and $a_t = \pi(s_t, \mathbf{b}_t)$:
			\begin{align}
				v^{\pi}(s_0, \mathbf{b}_0) 	&=	 E \bigg[ \sum_{t=0}^{\infty} \beta^t
				c(s_t, y_{t+1}, a_t) \vert s_0, \mathbf{b}_0 \bigg] \\
				&=	 E \bigg[ \sum_{t=0}^{\infty} \beta^t
				\left[  K(\mathbf{b}_t, a_t) + L(s_t, \mathbf{b}_t) \right]
				\vert s_0, \mathbf{b}_0 \bigg],
			\end{align}
		and where (3) follows from application of assumption (a). From assumption (b), $s_{t+1} \sim \gamma(a_t, y_{t+1})$, where $\gamma$ is a random variable depending only on $a_t$ and $y_{t+1}$. Then,
			\begin{align*}
				v^{\pi}(s_0, \mathbf{b}_0) 	&= 	\mathbb{E} \bigg[ \sum_{t=0}^{\infty} \beta^t \left[ K(\mathbf{b}_t, a_t) + 		
				L(s_t, \mathbf{b}_t)  \right] \vert s_0, \mathbf{b}_0 \bigg]		\\
				&=	 K(\mathbf{b}_0,a_0) + L(s_0, \mathbf{b}_0)  + 
				\mathbb{E} \bigg[ \sum_{t=1}^{\infty} \beta^t \big[
				K(\mathbf{b}_t, a_t) + L \big( \gamma(a_{t-1}, y_t), \mathbf{b}_t \big) \big]
				\vert s_0, \mathbf{b}_0 \bigg]	\\
				&= 	L(s_0, \mathbf{b}_0) + \mathbb{E} \bigg[ \sum_{t=0}^{\infty} \beta^t
				\big[ K(\mathbf{b}_t, a_t) +
				\beta L \big( \gamma(a_t, y_{t+1}), \mathbf{b}_{t+1} \big)  \big]
				\vert s_0, \mathbf{b}_0 \bigg]	\\
				&= 	L(s_0, \mathbf{b}_0) + \mathbb{E} \bigg[ \sum_{t=0}^{\infty} \beta^t
				\big[  K(\mathbf{b}_t, a_t) + \beta
				L \big( \gamma(a_t, y_{t+1}), \lambda(y_{t+1}, \mathbf{b}_t) \big)
				\big] \vert s_0, \mathbf{b}_0 \bigg]	\\
				&= 	L(s_0, \mathbf{b}_0) + \mathbb{E} \bigg[ \sum_{t=0}^{\infty} \beta^t
				\bigg[  K(\mathbf{b}_t, a_t) + \beta \sum_{y''} \sigma \big(y'' \vert \lambda(y_{t+1},\mathbf{b}_t) \big) \sum_{s'} 
				p(s' \vert y_{t+1}, a_t) L(s', y'')
				\bigg]
				\vert s_0, \mathbf{b}_0 \bigg]	\\
				&= 	L(s_0, \mathbf{b}_0) + \mathbb{E} \bigg[  \sum_{t=0}^{\infty} \beta^t 
				G(\mathbf{b}_t, a_t) \vert s_0, \mathbf{b}_0 \bigg]	\\
				&\geq L(s_0, \mathbf{b}_0) + \mathbb{E} \bigg[ \sum_{t=0}^{\infty} \beta^t 
				G(\mathbf{b}_t, a^*(\mathbf{b}_t)) \vert s_0, \mathbf{b}_0 \bigg].
			\end{align*}
		We conclude that the policy $\pi^*(s,\mathbf{b}) = a^*(\mathbf{b})$ for all $s \in \mathcal{S}$, $\mathbf{b} \in \mathcal{B}$ is stationary and optimal. 
	\end{proof}

	Now, we can prove the optimality of the base stock policy in Proposition \ref{prop: base stock optimality} by showing that it satisfies the conditions of Proposition \ref{Myopic} as a myopic optimal policy.
	
	\begin{proof}[Proof of Proposition \ref{prop: base stock optimality}.]
		We go case-by-case through the assumptions of Proposition \ref{Myopic}.
		\begin{enumerate}[label = (\roman*)]
			\item The cost function, $c$, for our inventory example comes from the following:
				\begin{equation*}
					\begin{split}
						\mathbb{E} \bigg[  \tilde{h} \left(\tilde{a} - \sum_{j=1}^{\tau} y_j \right)^+ &+ \tilde{p} \left(\sum_{j=1}^{\tau} y_j  - \tilde{a}\right)^+ \vert \mathbf{b} \bigg] \\
						&= \sum_{y_1, x'} \sigma(y_1, x' \vert \mathbf{b}) \underbrace{\mathbb{E}  \left[  \tilde{h} \left(\tilde{a} - \sum_{j=1}^{\tau} y_j \right)^+ + \tilde{p} \left(\sum_{j=1}^{\tau} y_j  - \tilde{a}\right)^+ \vert \mathbf{b}, y_1 \right] }_{\text{SEP-POMDP cost function, $c$}}.
					\end{split}
				\end{equation*}
				From this, we can see that $c$ is a function only of $\tilde{a}$ ($a$ in Proposition \ref{Myopic}) and the subsequent demand $y_1$ ($y'$ in Proposition \ref{Myopic}), and thus $(i)$ is satisfied with $c = K$.
			
			\item In the inventory position formulation, the dynamics of the inventory position are defined by the stochastic difference equation, $\tilde{s}_{t+1} = \tilde{a}_t - y_{t+1}$, and do not depend on $\tilde{s}_t$.
			
			\item This is the definition of the base stock levels in Equation \ref{eq: base stock policy definition}, where $c=K$ from $(i)$, above.
			
			\item This condition is guaranteed by the attainability condition of Proposition \ref{prop: base stock optimality}, namely: $a^*(\mathbf{b}) - y' \leq a^* \big( \lambda(y', x', \mathbf{b}) \big)$ for all $y', x', \mathbf{b}$.
		\end{enumerate} 
	Since $(i)-(iv)$ of Proposition \ref{Myopic} are satisfied, the base stock (myopic) policy defined by Equation \ref{eq: base stock policy definition} is optimal.
	\end{proof}

	% NEW APPENDIX SECTION
	\section{Computational Tractability}
	
	POMDPs are notoriously difficult to solve for other than small instances due to the fact that the belief space $\mathcal{B}$ contains an uncountably infinite number of possible belief vectors. There have been various approaches in the literature that seek to overcome the tractability issue of the POMDP. In this appendix we discuss additional types of solution procedures for POMDPs --- belief trajectory simulation, exact, information relaxation, and online heuristics --- and give examples of how the specialized structural properties of the SEP-POMDP can be utilized within these frameworks to solve (or approximately solve) SEP-POMDPs.

	\subsection{Belief Trajectory Simulation Methods}
	Belief trajectory simulation methods are based upon the intuition that, for many problems, there are only a small subset of beliefs that are \textit{reachable} under an optimal policy. Various approaches in the literature successively build a grid on $\mathcal{B}$ by alternating at each epoch between sampling new beliefs and performing value iteration operations on the new belief states (\cite{Pineau03}, \cite{Spaan05}).
	
	Here we present an $\textit{a priori}$ belief trajectory simulation method for constructing a discrete grid approximation, $\mathcal{B}' \subset \mathcal{B}$, which utilizes the \textit{actual dynamics} of the modulation and observation processes, while alleviating the computational burden associated with past approaches for the generalized POMDP due to the fact that learning in SEP-POMDPs is \textit{passive} and independent of control. This method turns solving the SEP-POMDP into solving a completely-observed MDP with state space $\mathcal{S} \times \mathcal{B}'$. %Then we present a $\textit{real-time}$ heuristic method that utilizes the SEP-POMDP structure to find approximate solutions for each belief state, as they are encountered in real-time. Finally, we discuss an information relaxation approach based on the lower bound in Proposition \ref{prop: information relaxation}.
	
	Suppose we have a metric space $(\mathcal{B}, \norm{\cdot})$, where $\norm{\cdot}$ is the sup-norm and $\mathcal{B}$ is the belief space. Let $\mathcal{B}_d \triangleq \left\{ \mathbf{b} \in \mathcal{B}: \exists \mathbf{b}' \in \mathcal{B}: \mathbf{b} = \frac{\lfloor \mathbf{b}' \cdot 10^d \rfloor}{10^d} \right\}$, the grid of points in $\mathcal{B}$ rounded to the $d$-th digit. Note that $\mathcal{B}_d  \subset \mathcal{B}$. We detail the so-called $\mathcal{B}'$ solution procedure for SEP-POMDPs.

	\begin{figure}[h]
		\hrulefill
		\centering
		\begin{enumerate}[start = 0]
			\item \textit{Initialization.} Initialize belief distribution, modulation state, number of simulation runs, mesh parameter, and cardinality parameter --- $\mathbf{b}_0$, $\mu_0$, $N$, $d$, and $K$ respectively.
			
			\item \textit{Belief simulation.} Generate, according to $P[y', \mu \vert \mu]$ the sequences $\{y_t, t = 1, \ldots, N\}$ and $\{\mu_t, t = 0, \ldots, N \}$. Then compute recursively $\{\mathbf{b}_t, t = 1, \ldots, N\}$ such that $\mathbf{b}_{t+1} = \lambda(y_{t+1}, \mathbf{b}_t)$ for $t = 0, \ldots, N-1$.
			
			\item \textit{$\mathcal{B}'$ definition.} Let $\tilde{\mathbf{b}}_t$ be $\mathbf{b}_t$ rounded to the $d$-th digit and let $\mathcal{B}' \triangleq \bigcup_{i=1}^K \tilde{\mathbf{b}}_{(i)}$, the $K$-th most frequently visited balls of radius $10^{-d}$ in $\mathcal{B}$.
			
			\item \textit{Solving the MDP with state space $\mathcal{S} \times \mathcal{B}'$.}  Solve the modified completely observed MDP with optimality equation
			\[\hat{v}(s, \mathbf{b}) = \min_{a \in \mathcal{A}(s)} \sum_{y'} \sigma(y' \vert \mathbf{b}) \bigg[ c(s, y', a) + \beta \sum_{s'} p(s' \vert y', s, a) \hat{v} \big( s', \mathbf{b}'(y', \mathbf{b}) \big) \bigg], \]
			where $\mathbf{b}'(y', \mathbf{b}) \approx \lambda(y', \mathbf{b})$ and $\mathbf{b}'(y', \mathbf{b}) \in \mathcal{B}'$.
		\end{enumerate}
		\caption{The $\mathcal{B}'$ method.}
		\hrulefill
	\end{figure}
	
	In step 0, we initialize the solution procedure. We note that $d$ should be a positive integer and controls the fineness of the grid. The cardinality parameter, $K$, determines how many points will be included in the approximate grid.
	
	In step 1, we simulate a trajectory of the beliefs by simulating the evolution of observations and modulation states according to the underlying Markov chain governing the dynamics, and recursively performing the belief update operations according to these observations and modulation states. So long as the Markov chain for the modulation states is ergodic, simulating one long trajectory should be sufficient for approximating a steady state distribution of modulation states. We note that this step is simulating a \textit{passive} learning environment since the belief updates are independent of control under the SEP-POMDP conditioning assumptions, guaranteeing that the learning operation for SEP-POMDPs is computationally tractable.
	
	In step 2, we determine $\{\tilde{\mathbf{b}}_t, t = 0, \ldots, N\}$, the set of simulated belief states rounded to the $d$-th digit, so that $\tilde{\mathbf{b}}_t$ is the unique point in $\mathcal{B}_d$ such that $\mathbf{b}_t$ is within a ball of radius $10^{-d}$ of $\tilde{\mathbf{b}}_t$. Let $\tilde{\mathcal{B}} = \bigcup_{t=1}^N \tilde{\mathbf{b}}_t$. (Note that $\tilde{\mathcal{B}} \subset \mathcal{B}_d$.) There is a complete order on $\tilde{\mathcal{B}}$ induced by the binary operator, $\preceq$, defined so that
	\[	\tilde{\mathbf{b}}_{(i) } \preceq \tilde{\mathbf{b}}_{(j) }	\Leftrightarrow i < j \text{ and } \sum_{t=1}^N \mathbf{1} \left\{ \norm{x_t - \tilde{\mathbf{b}}_{(i)}} \leq 10^{-d} \right\} \leq \sum_{t=1}^N \mathbf{1}\left\{ \norm{\mathbf{b}_t - \tilde{\mathbf{b}}_{(j)}} \leq 10^{-d} \right\}.	\]
	This order counts the number of simulated beliefs that are rounded to a particular $\tilde{\mathbf{b}}$ and ranks them. We then define $\mathcal{B}'$ to be the $K$-th most frequently visited rounded beliefs (Note that $\mathcal{B}' \subset \tilde{\mathcal{B}} \subset \mathcal{B}_d \subset \mathcal{B}$). Of course, $\mathcal{B}'$ has cardinality $K$, so it is finite in dimension.
	
	Finally, in step 3 we are left with the SEP-POMDP optimality equation, below
	\[v(s, \mathbf{b}) = \min_{a \in \mathcal{A}(s)} \sum_{y'} \sigma(y' \vert \mathbf{b}) \bigg[ c(s, y', a) + \beta \sum_{s'} p(s' \vert y', s, a) v \big( s', \lambda(y', \mathbf{b}) \big) \bigg], 	\quad \forall (s, \mathbf{b}) \in \mathcal{S} \times \mathcal{B}'. \]
	Our remaining challenge is that $\lambda(y', x)$ may not be in $\mathcal{B}'$ for a given $(y', x)$. Suppose $x \in \mathcal{B}'$. The hope is that $\exists \mathbf{b}'(y', \mathbf{b}) \in \mathcal{B}'$ such that $\lambda(y', \mathbf{b}) \approx \mathbf{b}'(y', \mathbf{b})$, and that $v(\cdot, \lambda(y',\mathbf{b})) \approx v(\cdot, \mathbf{b}'(y',\mathbf{b}))$. These assumptions may not hold if either $\lambda(y', \mathbf{b})$ is not near any point in $\mathcal{B}'$ (although intuitively, in most cases, it should be since we chose $\mathcal{B}'$ on the basis of frequently visited belief vectors in our simulation), or if $\lambda(y', \mathbf{b})$ is near a facet of the Sondik regions of $\mathcal{B}$, so that $v(\cdot, x'(y',x))$ is not a good approximation to $v(\cdot, \lambda(y',\mathbf{b}))$. There are many ways we could define $\mathbf{b}'(y', \mathbf{b})$, such as $\mathbf{b}'(y', \mathbf{b}) \triangleq \argmin_{\mathbf{b}' \in \mathcal{B}'} \{ \norm{\mathbf{b}' - \lambda(y', \mathbf{b})}\}$. 
	
	This creates a well-defined MDP, with state space $\mathcal{S} \times \mathcal{B}'$, which serves as our approximate model for the SEP-POMDP. The benefits of this method is that we reduce drastically the number of possible belief states that we need to consider in the SEP-POMDP by using the \textit{actual dynamics} of the system, which makes it better-suited than uniform or random grid methods for each particular problem instance (\cite{Lovejoy91}, \cite{Hauskrecht00}).

	\subsection{Exact Methods}
	Exact methods are based upon value iteration and seek to solve the POMDP exactly by utilizing the piecewise linear and concave structure of the value function with respect to $\mathbf{b}$ to construct the defining facets of the value function. \cite{Sondik78} and \cite{Sondik73} were the first to take this approach in their seminal papers. \cite{Kaelbling98} improved upon the complexity of this approach by using linear programming to construct the facet vectors. For the SEP-POMDP this structural result implies that if there is a finite set of vectors $\Gamma(s)$ for all $s$ such that $v(s,\mathbf{b}) = \min \{\mathbf{b} \gamma: \gamma \in \Gamma(s)\}$, then there is a finite set $\Gamma'(s)$ for all $s$ such that $Hv(s,\mathbf{b}) = \min\{\mathbf{b}\gamma: \gamma \in \Gamma'(s)\}$ and that in the limit, the fixed point of $H$, $v^*$, is concave in $x$ for all $s$. In analogy to computational procedures that make use of this structural characteristic for the POMDP, the process of constructing $\{\Gamma'(s)\}$ for the SEP-POMDP involves an intermediate step, the determination of the sets $\{\Gamma'(s,a)\}$ such that $\min \{\mathbf{b} \gamma: \gamma \in \Gamma'(s,a)\} = \sum_{y'} \sigma(y' \vert \mathbf{b}) h_{y'} \big( s, a, v(\cdot, \lambda(y',\mathbf{b})) \big)$. The computational implications of the inheritance property vary as a function of the structure under consideration and mirror the computational implications of this structure for the MDP analogs. For example, assume the MDP analogs are such that for each $y'$, there exists an optimal policy that is monotone in $s$. Then, for each $x$, there exists an optimal policy $\delta^*(s,\mathbf{b})$ such that if $s \leq s'$, then $\delta^*(s,\mathbf{b}) \leq \delta^*(s', \mathbf{b})$. It is therefore unnecessary to construct $\Gamma'(s',a)$ for all $a < \min \{\delta^*(s,\mathbf{b}): \mathbf{b} \in \mathcal{B}\}$.

	\subsection{Information Relaxation and Upper and Lower Bounds}
	Another common method for approximately solving stochastic dynamic programs is via information relaxation, as in \cite{Brown10}. We give a natural information relaxation-based heuristic here that is based on a relaxation of the partial-observability of the modulation process and can generate both upper and lower bounds on $v^*$. Suppose we want to minimize the expected total discounted cost, where at each decision epoch the DM has available the information as in the SEP-POMDP, $\mathscr{I}_t$, but also knowledge of the modulation states $\{\mu_t, \ldots, \mu_1\}$. Feasible policies map $\mathscr{I}_t \cup \{\mu_t, \ldots, \mu_1\}$ into feasible actions at all epochs $t$. The DM is faced with a MDP defined by the operator $H_M: V_M \mapsto V_M$, where $V_M$ is the space of bounded real-valued functions on $\mathcal{S} \times \mathcal{M}$,
	\begin{equation*}
	H_{M}v(s, \mu) 	= 	\min_{a \in \mathcal{A}(s)} \sum_{y', \mu'} P[y', \mu' \vert \mu] \left[ c(s, y', a) + \beta \sum_{s'}p(s' \vert y', s, a) v(s', \mu') \right].
	\end{equation*}
	
	In the following proposition, we show that the fixed point of $H_M$ can be used to determine a lower bound on $v^*$.
	
	\begin{proposition}
		$\sum_{\mu} x(\mu)v_M(s, \mu) \leq v^*(s, \mathbf{b})$ for all $(s, \mathbf{b}) \in \mathcal{S} \times \mathcal{B}$, where $v_M = H_M v_M$ and $v^* = Hv^*$.
		\label{prop: information relaxation}
	\end{proposition}
	
	Proof of the proposition follows by straightforward observation that all SEP-POMDP policies in $\Pi$ are feasible for this MDP, but not all policies for this MDP are feasible for the SEP-POMDP. We remark that this bound may be improved by applying a proper penalty term, akin to a Lagrangian relaxation, an idea developed in \cite{Brown10} and \cite{Rogers07}.
	
	The fixed point of $v^\pi$ of any policy $\pi$ can serve as an upper bound on $v^*$, where $v^\pi$ is determined exactly or approximated by simulation. If $v^{\pi_h} - v_M$ is small, then $\pi$ is a good sub-optimal policy. As an example, let $\pi_M: \mathcal{S} \times \mathcal{M} \mapsto \mathcal{A}$ be an optimal policy for the MDP having operator $H_M$. We remark that $\pi_M$ is determined when the lower bound presented in Proposition \ref{prop: information relaxation} is computed. Let $\pi$ be the randomized policy $\pi(s,\mathbf{b}) = \pi_M(s, \mu)$ with probability $\mathbf{b}(\mu)$. We would expect this policy to be an excellent sub-optimal policy if observations of the modulation process were highly accurate. As another example, if $\pi_{\mathcal{B}'}$ is the optimal policy generated for the MDP in Step 3 of Figure 2 (a function from $\mathcal{S} \times \mathcal{B}'$ to $\mathcal{A}$), then one might consider $\pi_h(s, \mathbf{b}) = \pi_{\mathcal{B}'}(s, \bar{\mathbf{b}})$, where $\bar{\mathbf{b}} = \argmin_{\mathbf{b}' \in \mathcal{B}'} \norm{ \mathbf{b} - \mathbf{b}'}$.

	\subsection{Heuristic Solution Procedure}
	We now present an alternative, heuristic solution procedure that must be implemented in an online manner. The fundamental idea is to map the SEP-POMDP into a related completely observed MDP with a state space on $\mathcal{S} \times \mathcal{Y}$ rather than on $\mathcal{S} \times \mathcal{B}$. We may assume that $\mathcal{Y}$ is finite in its cardinality, and thus this mapping is a state space dimensionality reduction technique (as is the $\mathcal{B}'$ procedure, above). The tradeoff is that we must solve such an MDP at each time epoch in order to capture the belief dynamics.
	\begin{figure}[h]
		\hrulefill
		\centering
		\begin{enumerate}[start = 0]
			\item \textit{Initialization.} Assume $(s_0, \mathbf{b}_0)$ is given. Set $t = 0$.
			
			\item Solve the completely observed MDP for all $(s, y')$:
			\[	v'_{y'}(s, \mathbf{b}_t)	=	\min_{a \in A(s)} \left\{ c(s, y', a) + 
			\beta \sum_{s'}p(s' \vert y', s, a) \sum_{y''} \sigma(y'' \vert \lambda(y', \mathbf{b}_t)) v'_{y''}(s', \mathbf{b}_t) \right\}.	\]
			Let $\delta^*_{y'}(s, \mathbf{b}_t)$ be an optimal policy, mapping $\mathcal{S} \times \mathcal{Y}$ into $A$.
			
			\item Choose action $a_t$ to equal $\delta^*_{y'}(s_t, \mathbf{b}_t)$ with probability $\sigma(y' \vert \mathbf{b}_t)$.
			
			\item Observe the observation $y_{t+1}$ (which will equal $y'$ with probability $\sigma(y' \vert \mathbf{b}_t)$). Set $\mathbf{b}_{t+1} = \lambda(y_{t+1}, \mathbf{b}_t)$.
			
			\item Observe the state $s_{t+1}$ (which will equal $s'$ with probability $p(s' \vert y_{t+1}, s_t, a_t)$).
			
			\item Increment $t \gets t+1$; go to 1.
		\end{enumerate}
		\caption{Real-time heuristic method.}
		\label{heuristic procedure}
		\hrulefill
	\end{figure}
	
	The intuition behind the procedure begins with the observation of the following inequality
	\[	\min_{a \in \mathcal{A}(s)} \sum_{y'} \sigma(y' \vert x) h\big(s, a, v(\cdot, \lambda(y', \mathbf{b})) \big)	
	\geq \sum_{y'} \sigma(y' \vert \mathbf{b})  \min_{a \in \mathcal{A}(s)}  h\big(s, a, v(\cdot, \lambda(y', x)) \big).		\]
	By pulling the minimization inside the summation, the idea is to establish a lower bound on $v^*$ by solving a related problem. We formalize this intuition in the subsequent proposition. Let 
	\[	\tilde{H}_{y'} \tilde{v}(s,\mathbf{b})	=	\min_{a \in \mathcal{A}(s)}	\left\{ c(s, y', a) 	+ 	
	\beta \sum_{s'} p(s' \vert y', s, a) \sum_{y''} \sigma(y'' \vert \lambda(y', \mathbf{b})) \tilde{v}_{y''} \big(s', \lambda(y', \mathbf{b}) \big) \right\},	\]
	and let $\tilde{v}_{y'}$ be the unique fixed point of $\tilde{H}_{y'}$.
	
	\begin{proposition}
		$v^*(s,\mathbf{b}) 	\geq \sum_{y'} \sigma(y' \vert \mathbf{b}) \tilde{v}_{y'}(s, \mathbf{b})$, for all $(s, \mathbf{b}) \in \mathcal{S} \times \mathcal{B}$.
		\label{prop: heuristic method bound}
	\end{proposition}
	
	Solving for $\{\tilde{v}_{y'}: z \in \mathcal{Y}\}$ is no more computationally tractable than solving for $v^*$ due to the cardinality of $\mathcal{B}$ and the dependence of $\tilde{v}_{y'}$ on $\lambda(y', \mathbf{b})$. In developing our heuristic procedure, we seek an approximation to $\{\tilde{v}_{y'}: z \in \mathcal{Y}\}$ for a fixed $x$. If we assume $\max_{y'} \norm{x - \lambda(y', x)}$ is small, then it is reasonable to assume that $\tilde{v}_{y''}(s', \lambda(y',x))$ is close to $\tilde{v}_{y''}(s', \mathbf{b})$ in many cases. This is effectively a \textit{learning rate} assumption (that learning is incremental and gradual), and is one that has been made in the literature, \textit{e.g.} \cite{Malladi17}. We then define a completely observed MDP with state space $S \times \mathcal{Y}$:
	\begin{equation} 
	v'_{y'}(s, \mathbf{b}) = \min_{a \in \mathcal{A}(s)} \left\{ c(s, y', a) + \beta \sum_{s'}p(s' \vert y', s, a) 
	\sum_{y''}\sigma(y'' \vert \lambda(y', \mathbf{b})) v'_{y''}(s', \mathbf{b})  \right\}
	\label{eq: approximate lower bound}
	\end{equation}
	This is the intuition behind step 2 in Figure \ref{heuristic procedure}. Since this approximation is for a fixed $x$, it is amenable to an online implementation, where this completely observed MDP is solved for each $x_t$.
	
	We remark that the following is \textit{likely} to be a valid inequality (although not necessarily)
	\[	v^*(s,\mathbf{b}) 	\geq \sum_{y'} \sigma(y' \vert \mathbf{b}) v'_{y'}(s,\mathbf{b}),	\]
	where $v'_{y'}$ is the fixed point of Equation \ref{eq: approximate lower bound}. We use Equation \ref{eq: approximate lower bound} to develop a heuristic that, for a given $(s, \mathbf{b})$, chooses action $\delta^*_{y'}(s,\mathbf{b})$ (an optimal policy mapping $\mathcal{S} \times \mathcal{Y}$ into $\mathcal{A}$, for this approximate MDP) with probability $\sigma(y' \vert \mathbf{b})$. This randomized policy is a \textit{probability matching} heuristic.
\end{appendices}

\end{document}